\documentclass[11pt]{amsart}
\pdfoutput=1 
\allowdisplaybreaks 
\usepackage{caption}

\usepackage{hyperref}
\usepackage{amsmath,amsfonts,amssymb,verbatim}
\usepackage{graphicx}
\usepackage[english,italian]{babel}
\usepackage[justification=centering]{caption}
\usepackage{color}

\usepackage{amsthm,amsxtra,color,enumerate,datetime,xcolor,}
\usepackage{imakeidx}
\usepackage[normalem]{ulem}
\usepackage{cancel}

\newtheorem{theorem}{Theorem}[section]
\newtheorem{corollary}[theorem]{Corollary}
\newtheorem{proposition}[theorem]{Proposition}
\newtheorem{lemma}[theorem]{Lemma}
\newtheorem{definition}[theorem]{Definition}
\newtheorem{remark}[theorem]{Remark}

\numberwithin{equation}{section}

\def \bC {\mathbb C}
\def \bD {\mathbb D}

\def \bH {\mathbb H}

\def \bN {\mathbb N}

\def \bR {\mathbb R}
\def \bS {\mathbb S}
\def \bR {\mathbb R}
\def \bR {\mathbb R}

\def \bZ {\mathbb Z}

\def \cC {\mathcal C}
\def \cD {\mathcal D}

\def \cF {\mathcal F}
\def \cG {\mathcal G}
\def \cH {\mathcal H}

\def \cO {\mathcal O}
\def \cP {\mathcal P}
\def \cQ {\mathcal Q}
\def \cR {\mathcal R}
\def \cS {\mathcal S}
\def \cR {\mathcal R}
\def \cR {\mathcal R}

\def \fh {\mathfrak h}

\def \fk {\mathfrak k}

\def \fn {\mathfrak n}
\def \fo {\mathfrak o}

\def \fr {\mathfrak r}
\def \fs {\mathfrak s}
\def \ft {\mathfrak t}
\def \fu {\mathfrak u}
\def \fv {\mathfrak v}
\def \fw {\mathfrak w}

\def \fz {\mathfrak z}

\def \fS {\mathfrak S}

\def \fU {\mathfrak U}

\def \fX {\mathfrak X}

\def \IM {\text{\rm Im}\,}
\def \al {\alpha}
\def \la {\lambda}
\def \ph {\varphi}
\def \del {\delta}

\def \lan {\langle}
\def \ran {\rangle}
\def \de {\partial}
\def \trans{\,{}^t\!}
\def \half{\frac12}
\def \inv{^{-1}}

\def \supp {\text{\rm supp\,}}

\def \dim {\text{\rm dim\,}}

\def \tr {\text{\rm tr\,}}
\def \id {\text{\rm id\,}}

\newcommand{\stkout}[1]{\ifmmode\text{\sout{\ensuremath{#1}}}\else\sout{#1}\fi}

\begin{document}

\selectlanguage{english}

\title[]
{Nilpotent Gelfand pairs\\ and Schwartz extensions of spherical transforms\\ via quotient pairs}

\author[]
{V\'eronique Fischer, Fulvio Ricci, Oksana Yakimova}

\address{University of Bath\\Claverton Down\\Bath BA2 7AY\\United Kingdom}
\email{v.c.m.fischer@bath.ac.uk}

\address {Scuola Normale Superiore\\ Piazza dei Cavalieri 7\\ 56126 Pisa\\ Italy}
\email{fricci@sns.it}

\address {Mathematisches Institut \\
Friedrich-Schiller-Universit\"at Jena \\ 07737 Jena \\ Germany}
\email{oksana.yakimova@uni-jena.de}



\begin{abstract}
It has been shown \cite{ADR1, ADR2, FR, FRY1} that for several nilpotent Gelfand pairs $(N,K)$ (i.e., with $N$ a nilpotent Lie group, $K$ a compact group of automorphisms of $N$ and the algebra $L^1(N)^K$ commutative) the spherical transform establishes a 1-to-1 correspondence between the space $\cS(N)^K$ of $K$-invariant Schwartz functions on $N$ and the space $\cS(\Sigma)$ of functions on the Gelfand spectrum $\Sigma$ of $L^1(N)^K$ which extend to Schwartz functions on $\bR^d$, once $\Sigma$ is suitably embedded in $\bR^d$. We call this {\it property (S)}.

We present here a general bootstrapping method which allows to establish property (S)  to new nilpotent pairs $(N,K)$, once the same property is known for a class of {\it quotient pairs} of $(N,K)$ and a $K$-invariant form of Hadamard formula holds on $N$.

We finally show how our method can be recursively applied to prove property (S) for a significant class of nilpotent Gelfand pairs.
\end{abstract}

\maketitle

\makeatletter
\renewcommand\l@subsection{\@tocline{2}{0pt}{3pc}{5pc}{}}
\makeatother

\tableofcontents

\section{Introduction}

\bigskip

\subsection{Presentation}

In this paper, we present our progress in the investigations of spectral properties 
of the abelian convolution  algebra  associated with a nilpotent Gelfand pair.
The properties we are interested in 
can be rephrased 
with essentially two different equivalent viewpoints:
on one hand in terms of functions defined on the spectrum of this algebra 
and on the other hand as spectral multipliers in several operators.
In the Euclidean setting, this equivalence follows readily from the basic properties of the Fourier transform and is commonly used when studying or using Euclidean Fourier multipliers.
Nilpotent Gelfand pairs provide a natural setting and a relatively explicit model 
where these questions can be investigated further. 

Let $N$ be a connected and simply connected nilpotent Lie group and $K$ a compact group of automorphisms of $N$. We say that $(N,K)$ is a {\it nilpotent Gelfand pair} if the convolution algebra $L^1(N)^K$ of $K$ -invariant integrable functions on $N$ is commutative. This is equivalent to saying that $(K\ltimes N,K)$ is a Gelfand pair in the ordinary sense \cite{Fa, W}. It must be mentioned from the very beginning that the group $N$ in a nilpotent Gelfand pair cannot  \cite{BJR90} have step greater than 2.

In this paper we address the following problem concerning smoothness preserving  properties of the spherical transform $\cG$ on $(N,K)$. 

It has been established in \cite{FeRu} that the Gelfand spectrum of $L^1(N)^K$ can be identified with a closed subset of $\bR^d$ for some $d$. In fact, there are many such identifications, depending on the choice of a generating set $\cD=(D_1,\dots,D_d)$ for the algebra $\bD(N)^K$ of left- and $K$-invariant differential operators on $N$, with $D_j$ symmetric for every $j$.

This identification consists in associating to each bounded spherical function $\ph$ the $d$-tuple $\big(\xi_1(\ph),\dots,\xi_d(\ph)\big)$ of eigenvalues of $\ph$ under $D_1,\dots,D_d$. The closed set $\Sigma_\cD$ of these $d$-tuples is then a homeomorphic image of the Gelfand spectrum $\Sigma$ and is also the joint $L^2$-spectrum of (the self-adjoint extensions of) the $D_j$.

Examples of  immersions of Gelfand spectra in $\bR^d$ are 
\begin{itemize}
\item the orbit space $\bR^n/K$, with $N=\bR^n$ and $K\subset O_n$, realised as the image of $\bR^n$ in $\bR^d$ under the polynomial map $P(x)=\big(p_1(x),\dots,p_d(x)\big)$, for a given choice of the Hilbert basis $(p_1,\dots,p_d)$ of $\bR[\bR^n]^K$;
\item the Heisenberg fan in $\bR^2$, with $N=H_n$, the $(2n+1)$-dimensional Heisenberg group, $K=U_n$ and $\cD$ consisting of the sub-Laplacian and the central derivative.
\end{itemize}

The problem we address is whether the spherical transform maps injectively the $K$-invariant Schwartz space $\cS(N)^K$ onto the space $\cS(\Sigma_\cD)$ of restrictions to $\Sigma_\cD$ of Schwartz functions on~$\bR^d$. We call this {\it property (S)}; see Section \ref{subsec_S} for the precise formulation.

Property (S) is known to hold for all pairs in which $N$ equal to $\bR^n$ or to a Heisenberg group  and for pairs where the action of $K$ on the centre of $\fn$ is rank-one, i.e., orbits are full spheres \cite{ADR1, ADR2, FR, FRY1}. 

The scope of this paper is to provide a general   inductive scheme which  allows to deduce  property (S) on more complex pairs from its validity on simpler ones. 

One feature of our approach to the general problem is that it does not rely on explicit formulas for spherical functions or for the differential operators involved, but rather on a bootstrapping procedure. Explicit formulas of the kind mentioned above have been found only for special pairs with $N$ equal to the Heisenberg group \cite{FW}. 

We first prove, in Section \ref{sec_procedure},  that property (S) is ``hereditary'' under the following operations:
\begin{itemize}
\item {\it central reductions} of $N$, obtained by replacing $N$ with a quotient modulo a $K$-invariant central subgroup;
\item {\it normal extensions} of $K$, obtained by replacing $K$ by a larger compact group $K^\sharp$ of aurtomorphisms of $N$, in which $K$ is normal;
\item taking {\it direct products} $(N_1\times N_2,K_1\times K_2)$ of pairs $(N_1,K_1)$, $(N_2,K_2)$  both with property~(S).
\end{itemize}

The central part of the paper is devoted to proving our main result, Theorem \ref{thm_main}, stating that, for a general nilpotent Gelfand pair,  property (S) holds provided  two   hypotheses are verified:
\begin{enumerate}
\item[(i)] property (S) holds for a family of lower-dimensional ``quotient pairs'';
\item[(ii)] there is a Hadamard-type expansion formula for spherical transforms at points of a ``singular subset'' of $\Sigma_\cD$; we call this {\it hypothesis (H)}.
\end{enumerate}

The notion of quotient pair and hypothesis (H) will be presented and studied in Sections \ref{sec_quotient_pairs} and \ref{sec_checkN+H} respectively.

In the last part of the paper we apply the previous results to prove property (S) for the following pairs\footnote{The table gives the Lie algebra $\fn$ of $N$ split as $\fn=\fv\oplus\fw$ with $\fw=[\fn,\fn]$ and $\fv$ $K$-invariant. The limitations on $n$ in the right-most column are meant to avoid superpositions with  already known cases.
For non-obvious notation and the explicit forms of the Lie brackets, see Section 7.} with a higher rank action of $K$ on the centre of $N$:

\begin{equation}\label{table_block1+2}
\begin{array}{|r|c|r@{\quad\oplus\quad }l|l|}
\hline
&K&
\fv &\fw&\\
\hline
1&{\rm SO}_n&\bR^n&\fs\fo_n&n\ge4\\
2&{\rm U}_n&\bC^n&\fu_n&n\ge2\\
3&\text{Sp}_n&\bH^n&(HS^2_0\bH^n\oplus\IM\bH)&n\ge2\\
4&{\rm SU}_{2n+1}&\bC^{2n+1}&\Lambda^2\bC^{2n+1}&n\ge1 \\
5&{\rm U}_{2n+1}&\bC^{2n+1}&(\Lambda^2\bC^{2n+1}\oplus\bR)&n\ge1\\
6& {\rm SU}_{2n}& \bC^{2n}&( \Lambda^2\bC^{2n}\oplus\bR)&n\ge2\\
\hline
\end{array}
\end{equation}

\medskip

These pairs appear in Vinberg's classification \cite{V2} of nilpotent Gelfand pairs with the two properties that $\fv$ is irreducible under $K$ and that they cannot be obtained from others by means of normal extensions or central reductions.

The pairs in Table \eqref{table_block1+2} do not exhaust the full Vinberg classification. As to the other Vinberg pairs, we have a proof that they also have property (S), but we  do not include the proof in this paper in order to avoid  extra  arguments which at the moment only apply to individual cases\footnote{Hypothesis (H) has been proved to hold for all Vinberg pairs in \cite{FRY2}. A proof of property (S) for the remaining Vinberg pairs will appear elsewhere.}. 
One advantage in restricting ourselves to the pairs in Table \eqref{table_block1+2} is that this family of pairs is self-contained, in the sense that the quotient pairs that intervene belong to the same class or have rank-one action on the center.

\subsection{The state of the art about property (S)}

We have already mentioned that, given self-adjoint operators $D_1,\dots,D_d$ generating $\bD(N)^K$, their joint $L^2$ spectrum is the set $\Sigma_\cD\subset\bR^d$ of $d$-tuples of their eigenvalues on the bounded spherical functions, cf. Section \ref{subsubsec_sph+cG}.  
One of the two implications in property (S) is already known, namely
$$
m\in\cS(\Sigma_\cD)\ \Longrightarrow\ \cG\inv m\in\cS(N)^K.
$$

This follows from the fact that $\cG\inv m$ is the convolution kernel $K_m$ of the operator $m(D_1,\dots,D_d)$, i.e. such that
$$
m(D_1,\dots,D_d)f=f*K_m\ ,
$$
and that the $D_j$ can be assumed to be Rockland operators; see Section \ref{subsec_S}, in particular Theorem \ref{hulanicki}.

Hence the crucial question is the opposite implication in property (S): given a $K$-invariant Schwartz function $F$ on $N$, does $\cG F$ extend from $\Sigma_\cD$ to a Schwartz function on $\bR^d$? It must be noticed that the answer does not depend on the choice of the system $\cD$ of differential operators (Proposition \ref{indipendence}).

In the most elementary case of $N=\bR^n$, with $K\subset O_n$, the (positive) answer follows directly from the G. Schwarz theorem \cite{Schw} about representability of $K$-invariant $C^\infty$-functions on $\bR^n$ as $C^\infty$-functions on a system of fundamental invariants, cf. \cite{ADR2} for a detailed discussion.

The next level of complexity is that of the Heisenberg group $H_n\cong \bC^n\times\bR$ as $N$,  $K$ being a closed subgroup of $U_n$ giving the required commutativity of $L^1(N)^K$. The method of proof, given in the general case in \cite{ADR2}, is an early use of the hypotheses (i), (ii) above. The quotient pair in (i) is in this case $(\bC^n,K)$, while the singular subset of $\Sigma_\cD$ in (ii) is the set of spherical functions with eigenvalue 0 relative to the central derivative $\de_t$. The Hadamard expansion is then done in the spectral variable $\la$ corresponding to $\de_t$. One proves that, for $F\in\cS(H_n)^K$ and every $N$,
$$
\cG F(\xi,\la)=g_0(\xi)+g_1(\xi)\la+\cdots g_N(\xi)\la^N+\la^{N+1}\cG F_N(\xi,\la)\ ,
$$
with $g_0,\dots,g_N\in\cS(\bR)$ and $F_N\in\cS(H_n)^K$. Here $\xi$ denote the $(d-1)$-tuple of remaining spectral coordinates.

When we pass to step-two groups different from the Heisenberg group, i.e., with a higher dimensional centre, matters become more complicated because two aspects must be taken into account: the action of $K$ on the centre $\fz$ of $\fn$ is no longer trivial in general, and the derived algebra $\fw=[\fn,\fn]$ may be properly contained in $\fz$. These elements have an effect on the structure of the Gelfand spectrum $\Sigma$, since the bounded spherical functions (which are all of positive type by \cite{BJR90}) are coefficients of the irreducible unitary representations on $N$, for which one of the parameters is given by the $K$-orbits in $\fw$.

This justifies that the first level of complexity beyond Heisenberg groups is that of groups with a rank-one action of $K$ on  $\fz$, for which there is only a one-parameter family of orbits in $\fz$ (and hence $\fw=\fz$). More elaborated applications of hypotheses (i) and (ii) to prove condition (S) on these pairs are in \cite{FR,FRY1}.

\subsection{Plan of the paper}

The paper is organised as follows.

In Section \ref{sec_generality}, we  set the basic notation and recall the general facts on nilpotent Gelfand pairs and  property (S).
In Section \ref{sec_procedure} we prove the above mentioned hereditary properties of pairs with property (S).

In Section \ref{sec_quotient_pairs} 
we  introduce the quotient pairs appearing in hypothesis (i). 
For $t\in\fw$, we denote by $K_t$  the stabilizer of $t$ in $K$, by $T_t(Kt)$  the tangent space at $t$ to the $K$-orbit $Kt$, and by the central quotient $\fn_t=\fv\oplus\big(\fw/T_t(Kt)\big)$. To avoid trivialities, we assume that $t$ is not fixed by $K$. 

We then split $\fw$ as
$$
\fw_0\oplus\check\fw\ ,
$$
where $\check\fw$ is the space of $K$-fixed elements and $\fw_0$ is its $K$-invariant complement. To each $t\in\fw_0\setminus\{0\}$ we associate the {\it quotient pair} $(N_t,K_t)$. 

Using  the notion of slice for a linear action of a group on a vector space, we present the radialisation method for constructing $K$-invariant functions on $N$ from $K_t$-invariant functions on $N_t$.

In Section \ref{sec_towards} we prove that, under the only assumption that (i) holds, property (S) holds limited to $F$ in the space $\cS_0(N)^K$ of Schwartz functions which have vanishing moments of any order in the $\fw_0$-variables.

In Section \ref{sec_checkN+H}, 
we introduce the singular subset of the Gelfand spectrum mentioned in (ii), defined as  the set of spherical functions which have eigenvalue 0 relative to derivatives in the $\fw_0$-variables. This set is naturally identified with the Gelfand spectrum of $(\check N,K)$ with $\check\fn=\fv\oplus(\fw/\fw_0)$. 

We then
show how hypothesis (H) allows to remove the vanishing moments condition and obtain property~(S) under assumptions (i) and (ii) above.

In Section \ref{sec-application}, 
we prove that property (S) holds for the pairs in Table 1.

\subsection{Acknowledgments}  
The third author 
wishes to thank the Centro di Ricerca Matematica Ennio De~Giorgi
and the Scuola Normale Superiore in Pisa for regular invitations 
and D.\,Timashev for bringing the book of Bredon to her attention. 
Part of the work was carried out 
at the Max-Planck-Institute f\"ur Mathematik (Bonn), where all the three authors have been short-term guests. 
We would like to thank these institutes for their warm hospitality and for  providing an excellent stimulating environment.

\section{Generalities on nilpotent Gelfand pairs and property (S)}
\label{sec_generality}

This section is devoted to some generalities  regarding 
 nilpotent Gelfand pairs and property (S).
 In Section~\ref{subsec_def_ngp}, 
we recall the equivalent definitions of a nilpotent Gelfand pair $(N,K)$
and set some notation for $N$.
In Section \ref{subsec_spherical}, 
we give equivalent descriptions of spherical functions and Gelfand spectra.
 We discuss property (S) in Section \ref{subsec_S}.
We eventually recall the notions of symmetrisation and Hilbert basis in 
Sections~\ref{subsec_sym} and \ref{subsec_hilbert_base} respectively.

\subsection{Definitions and notation}
\label{subsec_def_ngp}

Let $N$ be a nilpotent, connected and simply connected Lie group, and let $K$ be a compact group of automorphisms of $N$.

\begin{definition}\label{nGp}
 $(N,K)$ is a {\it nilpotent Gelfand pair}  if either of the following equivalent conditions is satisfied:
\begin{enumerate}
\item[\rm(i)]  the convolution algebra $L^1(N)^K$ of integrable $K$-invariant functions on $N$ is commutative;
\item[\rm(ii)] the algebra $\bD(N)^K$ of left-invariant and $K$-invariant differential operators on $N$ is commutative;
\item[\rm(iii)] if $\pi$ is an irreducible unitary representation of $N$ and $K_\pi$ is the stabiliser in $K$ of the equivalence class of $\pi$, then the representation space $\cH_\pi$ decomposes under $K_\pi$ without multiplicities;
\item[\rm(iv)] same as {\rm (iii)}, for $\pi$ generic.
\end{enumerate}
\end{definition}

This is the same as  saying that $(K\ltimes N,K)$ is a Gelfand pair. With $\fn$ denoting the Lie algebra of $N$, we often write 
$(\fn,K)$ instead of $(N,K)$.

We recall the following  basic result from  \cite{BJR90}.

\begin{theorem}
If $(N,K)$ is a nilpotent Gelfand pair, $N$ has step at most 2.
\end{theorem}

 We can then split $\fn$ as the direct sum $\fv\oplus\fw$, where $\fw=[\fn,\fn]$ is the derived algebra and $\fv$ a $K$-invariant complement of it. 
As $N$ has step at most 2, $\fw$ is abelian and is contained in the centre $\fz$ of $\fn$.
We regard the Lie bracket on $\fn$ as a skew-symmetric bilinear map from $\fv\times\fv$ to $\fw$. 

Using the exponential coordinates to parametrise elements of $N$, the product on $N$ is expressed as an operation on $\fv\oplus\fw$, via the Baker-Campbell-Hausdorff formula
$$
(v,w)\cdot(v',w')=\Big(v+v',w+w'+\half[v,v']\Big)\ .
$$

Accordingly, 0 will denote the identity element of $N$.

 Introducing a  $K$-invariant  inner product $\lan\ ,\ \ran$  on $\fv\oplus\fw$ under which $\fv\perp\fw$, we identify $\fn^*$ with $\fn$ throughout the paper.

On $N$, as well as on its Lie algebra $\fn=\fv\oplus\fw$,  
we consider the automorphic \emph{(group) dilations} 
\begin{equation}
\label{eq_gr_dil}
\del\cdot(v,w)=(\del^\half v,\del w),
\qquad \delta>0.
\end{equation}
We say that a function $f$, resp.  a differential operator $D$  on $N$, is homogeneous of degree $\nu$ (with respect to the group dilations)
when 
$$
f(\delta \cdot (v,w)) = \delta ^\nu f(v,w), 
\quad \delta>0, \ (v,w)\in N,
$$
respectively
$$
\left(D (f(\delta \, \cdot \,)\right)(v,w) = \delta ^\nu D f(v,w), 
\quad \delta>0, \ (v,w)\in N, \ f\in \cC^\infty(N).
$$

However, the expressions ``degree of a polynomial'' and ``order of a differential operator'' will have the ordinary meaning.

We fix Lebesgue measures $dv$ on $\fv$ and $dw$ on $\fw$ and the Haar measure  $dx=dvdw$   on $N$.

\subsection{Spherical functions}
\label{subsec_spherical}

Harmonic analysis on Gelfand pairs is based on the notions of spherical function (with particular emphasis on the bounded ones) and spherical transform, see e.g. \cite{Fa}, \cite[Ch. IV]{He2}. 
In this section and the next, we recall  these notions in the context of nilpotent Gelfand pairs and the basic properties which will be relevant for us.

Let $(N,K)$ be a nilpotent Gelfand pair.
The spherical functions are the joint $K$-invariant functions  $\ph$ on $N$ which are eigenfunctions of all operators in $\bD(N)^K$ and take value 1 at the identity. Spherical functions are real-analytic. 
Given $D\in\bD(N)^K$ and a spherical function~$\ph$, we denote by $\xi(D,\ph)$ the corresponding eigenvalue.
Hence a smooth $K$-invariant function $\ph$ on $N$ is a spherical function whenever
$$
\forall D\in \bD(N)^K\qquad
\exists \xi(D,\ph)\in \bC\qquad
D \ph = \xi(D,\ph) \ \ph, 
\quad \mbox{and}\quad\ph(0)=1.
$$

The bounded spherical functions are  characterised by the multiplicative identity
$$
\int_Nf*g(x)\ph(x\inv)\,dx=\Big(\int_Nf(x)\ph(x\inv)\,dx\Big)\Big(\int_Ng(x)\ph(x\inv)\,dx\Big)
$$
for all $f,g\in L^1(N)^K$.

It has been proved in \cite{BJR90} that all bounded spherical functions of  $(N,K)$ are of positive type. Hence they are in one-to-one correspondence with (equivalence classes of) irreducible unitary representations of $K\ltimes N$  admitting non-trivial $K$-invariant vectors and can be expressed as diagonal matrix entries of such representations. 
However, we prefer to view these expressions  as  partial traces of  irreducible unitary representations of $N$, see \eqref{trace} below.

For $\zeta\in\fw$, denote by $\fr_\zeta\subseteq\fv$ the radical of the bilinear form $B_\zeta(v,v')=\lan\zeta,[v,v']\ran$.
The following statement is a direct consequence of the Stone-von Neumann theorem and we omit its proof.

\begin{lemma}\label{representations}
For each $\zeta\in\fw$ 
there is a unique, up to equivalence, irreducible unitary representation $\pi_\zeta$ of $N$ such that $d\pi_\zeta(0,w)=i\lan\zeta,w\ran I$ for all $w\in \fw$ and $d\pi_\zeta(v,0)=0$ for all $v\in \fr_\zeta$. 

For each $\zeta\in\fw$ and $\omega\in\fr_\zeta$,  define the irreducible representation $\pi_{\zeta,\omega}$ by the condition
 \begin{equation}\label{zeta-omega}
d\pi_{\zeta,\omega}(v,w)=d\pi_\zeta(v,w)+i\lan v,\omega\ran I\ .
\end{equation}
Every irreducible unitary representation of $N$ is equivalent to one, and only one, $\pi_{\zeta,\omega}$.
\end{lemma}

We denote by $\cH_\zeta$ the representation space of the representations $\pi_{\zeta,\omega}$. 
The stabiliser $K_{\zeta,\omega}\subset K$ of the point $\omega+\zeta\in\fn$ also stabilises the equivalence class of $\pi_{\zeta,\omega}$,  giving rise to a unitary representation\footnote{In general, this operation leads to a projective representation of the stabiliser in $K$. In our case we obtain true representations, since restriction of the metaplectic representation of ${\rm Sp}(\fr_\zeta^\perp,B_\zeta)$ to a compact subgroup can be linearised \cite{F}.}  $\sigma$ of $K_{\zeta,\omega}$ on $\cH_\zeta$ such that
$$
\pi_{\zeta,\omega}(kv,kw)=\sigma(k)\pi_{\zeta,\omega}(v,w)\sigma(k)^{-1}
$$
for $k\in K_{\zeta,\omega}$.

We have the following characterisation, cf.  \cite{C, V1}.

\begin{proposition}\label{carcano}
$(N,K)$ is a nilpotent Gelfand pair if and only if, for each  (resp. for generic) $\zeta,\omega$, $\cH_\zeta$ decomposes without multiplicities into irreducible components under the action of $K_{\zeta,\omega}$.
\end{proposition}

Hence, if $(N,K)$ is a nilpotent Gelfand pair,
\begin{equation}\label{multiplicity-free}
\cH_\zeta=\sum_{\mu\in \fX_{\zeta,\omega}}V(\mu)\ ,
\qquad \mbox{with} \quad
\fX_{\zeta,\omega}\subseteq \widehat{K_{\zeta,\omega}}.
\end{equation}
 To each $\mu\in\fX_{\zeta,\omega}$ we can associate the spherical function
\begin{equation}\label{trace}
\ph_{\zeta,\omega,\mu}(v,w)=\frac1{\dim V(\mu)}\int_K \tr\big(\pi_{\zeta,\omega}(kv,kw)_{|_{V(\mu)}}\big)\,dk\ .
\end{equation}

For a given $k\in K$, we have $\fX_{k\zeta,k\omega}=\fX_{\zeta,\omega}$, under the natural identification of the dual object $\widehat{K_{\zeta,\omega}}$ of $K_{\zeta,\omega}$ with the dual object of $K_{\pi_{\zeta,\omega}^k}=k\inv K_{\pi_{\zeta,\omega}} k$, and
\begin{equation}
\label{eq_ph_Kequivariance}
\ph_{k\zeta,k\omega,\mu}=\ph_{\zeta,\omega,\mu}\ .
\end{equation}
Up to this $K$-equivariance, 
the parametrisation of the spherical function via $\ph_{\zeta,\omega,\mu}$ is unique.

\subsection{Gelfand spectrum and spherical transform}
\label{subsubsec_sph+cG}

The \emph{Gelfand spectrum} of the nilpotent Gelfand pair~$(N,K)$ is the 
spectrum  of the commutative convolution algebra $L^1(N)^K$,

$$
\Sigma(N,K)=\big\{\ph: \ph\text{ bounded spherical function}\big\}
$$
 with the compact-open topology. We will  denote it just by $\Sigma$ if there is no ambiguity.
 
The \emph{spherical transform} of a function $F\in L^1(N)^K$ is defined via
\begin{equation}\label{transform}
\cG F(\ph)=\int_N F(x)\ph(x\inv)\,dx,
\quad \ph\in \Sigma.
\end{equation}
This yields a continuous linear one-to-one mapping 
$\cG:L^1(N)^K\longrightarrow C_0(\Sigma)$  which transforms convolution into pointwise multiplication.

\medskip

The last part of the previous section shows that  $\Sigma(N,K)$ can be parametrised by  triples $(\zeta,\omega,\mu)$, precisely
\begin{equation}\label{Sigma-triples}
\Sigma(N,K)\cong \big\{(\zeta,\omega,\mu):\zeta\in\fw,\omega\in\fr_\zeta,\mu\in\widehat{K_{(\zeta,\omega)}}\big\}/K,
\end{equation}
where the action of $K$ is expressed by \eqref{eq_ph_Kequivariance}.

But a different kind of parametrisation provides topological embeddings of $\Sigma$ into Euclidean spaces in the following way. Let
$$
\cD=(D_1,\dots,D_d)\ ,
$$ 
be a $d$-tuple of essentially self-adjoint operators which generate $\bD(N)^K$ as an algebra. Such a tuple exists and form a family of strongly commuting self-adjoint operators  whose  joint $L^2$-spectrum can be identified with the Gelfand spectrum, see \cite{FR}. Indeed,
given $\ph\in\Sigma$, denote by $\xi_j(\ph)\in\bR$ the eigenvalue $\xi(D_j,\ph)$ of $\ph$ under $D_j$ for each $j=1,\ldots,j$.
Every bounded spherical function $\ph$ is identified by the $d$-tuple $\xi=\xi(\ph)=\big(\xi_1(\ph),\dots,\xi_d(\ph)\big)$ of eigenvalues of $\ph$ relative to $\cD$;
moreover the $d$-tuples $\xi(\ph)$ form a closed subset 
\begin{equation}\label{Sigma_D}
\Sigma_\cD=\big\{\xi(\ph)=\big(\xi_1(\ph),\dots,\xi_d(\ph)\big):\ph\in\Sigma\big\}
\end{equation}
 of $\bR^d$ which is homeomorphic to $\Sigma$, see \cite{FeRu}. 
 Consequently the spherical transform $\cG F$ in \eqref{transform} can be viewed as a function on $\Sigma_\cD$.
The particular choice of a $d$-tuple $\cD$ to realise of a Gelfand spectrum  
as $\Sigma_\cD \subset \bR^d$ is often irrelevant 
as there is always some polynomial mappings between two finite sets of generators of $\bD(N)^K$: 
\begin{lemma}
\label{lem_poly_map_SigmacDs}
Let $\cD=(D_1,\dots,D_d)$ and $\cD'=(D'_1,\dots,D'_{d'})$ be 
two tuples of operators generating the algebra $\bD(N)^K$.
Then there exist polynomials $P_j$, $j=1,\ldots,d$, and $Q_k$, $k=1,\ldots d'$ such that 
$D_j=P_j(\cD')$ for all $j=1,\ldots,d$,
and 
$D'_k=Q_k(\cD)$ for all $k=1,\ldots,d'$.
The maps 
$$
P=(P_1,\ldots, P_d)\ ,\qquad Q=(Q_1,\ldots, Q_{d'})
$$ 
are homeomorphisms of $\Sigma_{\cD'}$ onto $\Sigma_\cD$, resp. of $\Sigma_\cD$ onto $\Sigma_{\cD'}$, and are inverse of each other.
\end{lemma}

In the case where $N=\bR^n$ and $K$ is trivial, $\bD(\bR^n)^K$ is the algebra of all constant coefficient differential operators, and the bounded spherical functions are the unitary characters $\ph_\la(x)=e^{i\la\cdot x}$, for $\la\in\bR^n$. Taking
$$
\cD=\big(i\inv\de_{x_1},\dots,i\inv\de_{x_n}\big)\ ,
$$
we have $\xi(\ph_\la)=\la$, so that  $\Sigma_\cD=\bR^n$, and $\cG F=\hat F$ is the ordinary Fourier transform.

\subsection{Property (S)}
\label{subsec_S}

In this section, we give the precise formulation of property (S) and summarise the known results about it.

Let $\cS(N)^K$ 
denote the Fr\'echet space of $K$-invariant Schwartz function on $N$. Let
$$
\cS(\Sigma_\cD)\overset{\rm def}=\cS(\bR^d)/\{f:f_{|{\Sigma_\cD}}=0\}
$$
be the space of restrictions to $\Sigma_\cD$ of Schwartz functions on $\bR^d$, with the quotient topology. 
 Property (S) is stated as follows:
\begin{center}
(S) \,
\begin{tabular}{c}
{\it 
The spherical transform $\cG$ maps the Fr\'echet space $\cS(N)^K$ }
\\
{\it isomorphically onto $\cS(\Sigma_\cD)$.}
\end{tabular}
\end{center}

The following statement follows directly from Lemma \ref{lem_poly_map_SigmacDs}, see  \cite{ADR2} and~\cite{FR}:
\begin{proposition}\label{indipendence}
The validity of property (S) is independent of the choice
 of $\cD$.
\end{proposition}

Property (S) has been proved to hold in several cases.
For `abelian pairs', i.e., with $N=\bR^n$ and $K\subset{\rm GL}_n(\bR)$ compact, it has been shown in \cite{ADR2} that property (S) follows from G.~Schwarz's extension  \cite{Schw}  of Whitney's theorem \cite{Wh} to general linear actions of compact groups on $\bR^n$.

For nonabelian $N$, property (S) has been proved in the following cases:
\begin{enumerate}
\item[(i)]  pairs in which $N$ is a Heisenberg group or a complexified Heisenberg group \cite{ADR1, ADR2};
\item[(ii)]  ``rank-one'' pairs, where $\fw=\fz$, the centre of $\fn$, and the $K$-orbits in $\fw$ are full spheres~\cite{FR, FRY1}.
\end{enumerate}
\smallskip

The argument used in Section 5 of \cite{ADR1} also allows to obtain the following statement. We omit its proof, which only requires minor modifications.

\begin{proposition}\label{trivial}
Let $(N,K)$ be a nilpotent Gelfand pair where $K$ acts trivially on $\fw$. Then property $\rm{(S)}$ is satisfied.
\end{proposition}

As we have already mentioned in the introduction, one of the two implications which constitute property (S) is a matter of functional calculus on Rockland operators on graded groups (i.e., with a graded Lie algebra), due to the following equivalence for $K\in L^1(N)^K$:
$$
g(D_1,\dots,D_d)f=f*K\ \Longleftrightarrow\ \cG K=g_{|_{\Sigma_\cD}}\ .
$$

It was proved in \cite{H} that if $L$ is a Rockland operator and $g$ is a Schwartz function on the line, then the operator $g(L)$ is given by convolution with a Schwartz kernel. This statement has been later extended to commuting families of $d$ Rockland operators and  $g\in\cS(\bR^d)$, cf.~\cite{Ven} and  \cite[Th. 5.2]{ADR2}.  

Since on any nilpotent Gelfand pair, we always have a system $\cD$ consisting of Rockland operators \cite{ADR2} and 
this has the following consequence.

\begin{theorem}[\cite{ADR2, FR}]\label{hulanicki}
Let $(N,K)$ be a nilpotent Gelfand pair, and $\cD$, $\Sigma_\cD\subset\bR^d$  as  above. Given any Schwartz function $g$ on $\bR^d$, there is a $K$-invariant Schwartz function $F$ on $N$, depending continuously on $g$, such that $\cG F=g_{|_{\Sigma_\cD}}$.
\end{theorem}

In other words,  the continuous inclusion
\begin{equation}\label{inclusion}
\cS(\Sigma_\cD)\subseteq \cG\big(\cS(N)^K\big)
\end{equation}
holds for any nilpotent Gelfand pair.
As $\cG$ is linear and one-to-one, 
the open mapping theorem for Fr\'echet spaces \cite{T} implies that
the proof of property (S)  reduces to proving the opposite inclusion, 
 i.e., that the spherical transform of any function in $\cS(N)^K$,
viewed as a function on $\Sigma_\cD$, admits a Schwartz extension to~$\bR^d$.

\medskip

By Proposition \ref{indipendence}, 
property (S) is independent of the choice of $\cD$.
Convenient choices of $\cD$ will be obtained by applying the symmetrisation mapping to polynomials on $\fn$.

\subsection{Symmetrisation}
\label{subsec_sym}

In this section, we set some notation and recall some properties of the symmetrisation mapping.

The symmetrisation mapping can be defined for any Lie group $N$
and is completely independent of the notion of nilpotent Gelfand pair.
Denoting by $\fn$ the Lie algebra of $N$, 
the  symmetrisation mapping $\la=\la_N$
is the unique linear bijection 
 from the symmetric algebra $\fS(\fn)$ onto the universal enveloping algebra $\fU(\fn)$
 which satisfies the identity $\la(X^n)=X^n$ for every $X\in\fn$ and $n\in\bN$ \cite[Theorem 4.3 in Ch.II]{He2}. 
 The symmetric algebra $\fS(\fn)$ may be viewed as the space $\cP(\fn^*)$ of polynomials on the dual space $\fn^*$, 
and, after identifying $\fn^*$ and $\fn$, 
as the the space $\cP(\fn)$ of polynomials on $\fn$.  
When  the elements of $\fU(\fn)$ are regarded as left-invariant differential operators on $N$, we use the notation $\bD(N)$.

Following \cite[Sect. 2.2]{FRY1}, we will use a modified symmetrisation $\la'_N:\cP(\fn)\longrightarrow\bD(N)$, which maps the polynomial $p\in\cP(\fn)$ to the differential operator
\begin{equation}\label{modsym}
\la'(p)F=p(i\inv\nabla_{v'},i\inv\nabla_{w'})_{|_{v'=w'=0}} F\big((v,w)\cdot(v',w')\big)\ .
\end{equation}
i.e., $\la'(p)=\la\big(p(i\inv\cdot)\big)$, in terms of the standard symmetrisation $\la$.
Hence only constants are changed with this modification but 
 its advantage is that  polynomials with real coefficients are transformed by $\la'$ into formally self-adjoint differential operators\footnote{The first two authors take this opportunity to correct an error in the formulation of Proposition 3.1 in \cite{FR}: it applies to operators $D=\la'(p)$ with $p$ real.}.
 When it is necessary to specify the group $N$, we write $\la'_N$ instead of $\la'$.

In the next lemma, we summarise some properties of $\lambda'$ which are readily checked, like 
the compatibility with the usual and homogeneous degrees for polynomials and differential operators (Parts \eqref{item_lem_sym_degree} and \eqref{item_lem_sym_hom} resp.),
with the action of a compact group (Part \eqref{item_lem_sym_K}).
We also give some weak properties on symmetrisation of product of polynomials  in Parts \eqref{item_lem_sym_degree}, \eqref{item_lem_sym_hom} and
\eqref{item_lem_sym_pdtz}.

\begin{lemma}
\label{lem_sym}
We consider the symmetrisation $\lambda'$ on a Lie group $N$ as above. 
It is a linear isomorphism from the space $\cP(\fn)$ onto $\bD(N)$
which satisfies the following property:
\begin{enumerate}
\item 
\label{item_lem_sym_degree}
The degree of  $p\in \cP(\fn)$  is equal to the order of 
$\lambda'(p)$.
Furthermore, if $p_1,p_2\in \cP(\fn)$ are polynomials of degree $d_1,d_2$ respectively, 
 the differential operator $\lambda'(p_1p_2)-\lambda'(p_1)\lambda'(p_2)$ is of order $<d_1+d_2$.
\item 
\label{item_lem_sym_hom}
 With respect to the group dilations defined in \eqref{eq_gr_dil}, 
the homogeneous degree of the polynomial $p\in \cP(\fn)$  is equal to the homogeneous degree of the  differential operator $\lambda'(p)$.
 Furthermore, if $p_1,p_2\in \cP(\fn)$ have homogeneous degree $d_1,d_2$ respectively, 
 the differential operator $\lambda'(p_1)\lambda'(p_2)$ is homogeneous of order $d_1+d_2$.
\item 
\label{item_lem_sym_K}
Let $K$ be a compact group which acts by automorphism on the group $N$.
The map $\lambda':\cP(\fn) \to \bD(N)$ is equivariant for the natural actions of $K$, that is, 
$$
\lambda'(p\circ k)\ (f) = \lambda'(p) \ (f\circ k^{-1} ), 
\quad p\in \cP(\fn), \ k\in K, \ f\in \cC^\infty(N).
$$
Consequently, if the polynomial $p$ is $K$-invariant, then 
the operator $\lambda'(p)$ is $K$-invariant.
\item
\label{item_lem_sym_pdtz}
Let $\fz$ be the centre of the Lie algebra $\fn$.
 If $p_1\in \cP(\fz)$ and $p_2\in \cP(\fn)$ then 
$$
\lambda'(p_1p_2)=\lambda'(p_1)\lambda'(p_2)
\quad\mbox{and}\quad
\lambda'(p_1)=p_1(i\inv\nabla_{\fz}).
$$
\end{enumerate}
\end{lemma}

The proof of Lemma \ref{lem_sym} is left to the reader.

\medskip

Convenient choices of $\cD$ will be obtained by applying $\lambda'$ to certain families of polynomials on $\fn$, as explained in the next section.

\subsection{Hilbert bases}
\label{subsec_hilbert_base}

Let $(N,K)$ be a nilpotent Gelfand pair.

We say that a polynomial $p\in\cP(\fn)$ has {\it bi-degree} $(r,s)$ if $p\in\cP^r(\fv)\otimes\cP^s(\fw)$. Notice that the homogeneous degree of $p$ is then $\nu=\frac r2+s$.

By a {\it  bi-homogeneous Hilbert basis} of $(\fn, K)$, we mean a $d$-tuple ${\boldsymbol\rho}=(\rho_1,\dots, \rho_d)$ of real, $K$-invariant polynomials on $\fn$ generating the $K$-invariant polynomial algebra $\cP(\fn)^K$ over $\fn$ and with each $\rho_j$   has bi-degree $(r_j,s_j)$.

We set $D_j=\la'(\rho_j)$ and $\cD=(D_1,\dots,D_d)$.  Then $\cD$ generates $\bD(N)^K$ and each $D_j$ is homogeneous of degree $\nu_j=r_j+2s_j$ with respect to the group dilations.
It is not difficult to see that if $\ph\in\Sigma$, then also $\ph^\del(v,t)=\ph\big(\del\cdot(v,w)\big)$ is a bounded spherical function for every $\del>0$. 
Consequently we have:
 $$
 \xi(D_j,\ph^\del)=\del^{\nu_j}\xi(D_j,\ph)\ .
 $$
 
Hence $\Sigma_\cD$ is invariant under the anisotropic dilations on $\bR^d$
\begin{equation}\label{dilations}
\xi=(\xi_1,\dots,\xi_d)\longmapsto \big(\del^{\nu_1}\xi_1,\dots,\del^{\nu_d}\xi_d\big)=D(\del)\xi\ ,\qquad (\del>0)\ .
\end{equation}

\medskip

We split $\fw$ as
\begin{equation}\label{z_0}
\fw=\fw_0\oplus\check\fw\ ,
\end{equation}
where $\check\fw$ denotes the subspace of $K$-fixed elements of $\fw$, and  $\fw_0$ its (unique) $K$-invariant complement in $\fw$. 

We will privilege bi-homogeneous Hilbert bases ${\boldsymbol\rho}$  which split as ${\boldsymbol\rho}=({\boldsymbol\rho}_{\fw_0},{\boldsymbol\rho}_{\check\fw},{\boldsymbol\rho}_\fv,{\boldsymbol\rho}_{\fv,\fw_0})$
\begin{enumerate}
\item[(i)] ${\boldsymbol\rho}_{\fw_0}$ is a homogeneous Hilbert basis  of $\cP(\fw_0)^K$; 
\item [(ii)] ${\boldsymbol\rho}_{\check\fw}$ is any system  of coordinate functions on $\check\fw$;
\item[(iii)] ${\boldsymbol\rho}_\fv$ is a homogeneous Hilbert basis  of $\cP(\fv)^K$;
\item[(iv)] ${\boldsymbol\rho}_{\fv,\fw_0}$ consists of bi-homogeneous polynomials in $\big(\cP^{r_j}(\fv)\otimes\cP^{s_j}(\fw_0)\big)^K$ with $r_j,s_j>0$.
\end{enumerate}
 The existence of such Hilbert bases is obvious.
We denote by $\cD_{\check\fw}$, $\cD_{\fw_0}$, $\cD_\fv$, $\cD_{\fv,\fw_0}$  the families of operators corresponding to the corresponding subfamilies of polynomials in ${\boldsymbol\rho}$ 
via the symmetrisation mapping $\la'$.

Denoting by $d_{\fw_0}$, $d_{\check\fw}$, $d_\fv$, $d_{\fv,\fw_0}$  the number of elements in each subfamily, we split $\bR^d$ as
$$
\bR^d=\bR^{d_{\fw_0}}\times\bR^{d_{\check\fw}}\times\bR^{d_{\fv}}\times\bR^{d_{\fv,\fw_0}}\ ,
$$
and set
$$
\xi(\ph)=\big(\xi_{\fw_0}(\ph),\xi_{\check\fw}(\ph),\xi_\fv(\ph),\xi_{\fv,\fw_0}(\ph)\big)\ ,
$$
for the eigenvalues of a function $\ph\in \Sigma(N,K)$
for $\cD$.

Whenever a unified notation for all invariants on $\fw$ is preferable, we we will use the symbols ${\boldsymbol\rho}_\fw$, $d_\fw$, $\xi_\fw(\ph)$, etc.

In Section \ref{special},  we will need a different splitting of the family ${\boldsymbol\rho}$, which takes into account the degrees of polynomials in the $\fv$-variables  (in the ordinary sense).
Then ${\boldsymbol\rho}_{(k)}$ will denote the (possibly empty) set of elements which have degree $k$ in the $\fv$-variables. 
In particular, ${\boldsymbol\rho}_{(0)}={\boldsymbol\rho}_\fw$.
 In  this situation, the elements of ${\boldsymbol\rho}_{(k)}$ will be labelled as $(\rho_{k,1},\dots,\rho_{k,d_{(k)}})$.
 Accordingly,
we  split $\cD$ and the corresponding set of eigenvalues as 
$\cD=(\cD_{(k)})_{k\ge0}$
and $\xi=(\xi_{(k)})_{k\ge0} \in \bR^d$.

\section{ 
Hereditarity of property (S)}
\label{sec_procedure}

In this section, 
we list certain procedures   on the groups $N$ and/or $K$ which, if applied to pairs satisfying property (S), produce new ones also satisfying property (S). 
These operations are:
 normal extensions of $K$, direct products of pairs,
 central reductions\footnote{The expression ``central reduction'' is kept, here and in the sequel, as in \cite{V1} for linguistic convenience, though it only refers to subspaces of the derived algebra $\fw$. Nonetheless, Proposition \ref{central} still holds if $\fs$ is taken as a subspace of $\fz$.} of $N$. They are  stated below 
 in  Propositions  \ref{normal}, \ref{product} , and \ref{central}
 respectively.
 
  \begin{proposition}[Normal extension of $K$]\label{normal}
Let $K^\#$ be a compact group of automorphisms of a nilpotent Lie group $N$, and $K$ a normal subgroup of $K^\#$. 

If $(N,K)$ is a nilpotent Gelfand pair satisfying property (S), 
then $(N,K^\#)$ is also a nilpotent Gelfand pair satisfying property (S).
\end{proposition}

The proof of Proposition \ref{normal} was given in 
\cite{FRY1}  (thereby extending an argument of \cite{ADR2}).

\begin{proposition}[Direct product]\label{product}
If  $(N_1,K_1)$, $(N_2,K_2)$ are two nilpotent Gelfand pairs satisfying (S), 
then their direct product $(N_1\times N_2,K_1\times K_2)$ is also a nilpotent Gelfand pair satisfying (S).
\end{proposition}

\begin{proposition}[Central reduction]\label{central}
Let $(N,K)$ be a nilpotent Gelfand pair.
If $\fs$ is a  $K$-invariant subspace of $\fw$,
 denote by $N'$ the quotient group with Lie algebra $\fn / \fs$.

If the pair $(N,K)$ satisfies property (S) then $(N',K)$ is also a nilpotent Gelfand pair  satisfying property~(S).
 \end{proposition}

In Section \ref{subsec_pf_prop_product}, 
we give the proof of Proposition \ref{product}.
In Section \ref{subsec_quotient}, we introduce some notions attached with quotients of a nilpotent Gelfand pair which will be useful in the proof of Proposition \ref{central}
(given in Section \ref{subsec_pf_prop_central})
as well as in other parts of the paper.

Proposition \ref{central} has been announced in \cite{FRY2} without proof.

\subsection{Proof of Proposition \ref{product}}
\label{subsec_pf_prop_product}\ 
We will need the following property of decompositions of Schwartz functions on the product of two Euclidean spaces.
We use the following Schwartz norms on $\cS(\bR^n)$:
$$
\|f\|_{\cS(\bR^n), M}
=
\|f\|_{M}
=
\sup_{|\alpha|\leq M, x\in \bR^n} (1+|x|)^M |\partial^\alpha f(x)|
\quad , \quad M\in \bN
\ .
$$

\begin{lemma}\label{lemma_product}
Let $n_1$ and $n_2$ be two positive integers. Set $n=n_1+n_2$.

Let also $\psi_\nu$, $\nu=1,2$, be smooth function on $\bR^{n_\nu}$,
supported in $[-1, 1]^{n_\nu}$ and satisfying:
$$
0\leq \psi_\nu\leq 1
\qquad\mbox{and}\qquad
\psi_\nu=1 \;\text{\rm on }\; \big[-\frac 34, \frac34\big]^{n_\nu}
\ .
$$

For $\nu=1,2$, $l_\nu,m_\nu\in\bZ^{n_\nu}$, set
$$
H^{(\nu)}_{l_\nu,m_\nu}(x_\nu)=e^{ix_\nu\cdot m_\nu} 
\psi_\nu(x_\nu+l_\nu)
\ .
$$

Then the following properties hold.
\begin{enumerate}
\item[\rm(a)] Let $\nu=1$ or $2$.
Given $M\in \bN$ there is a constant $C_M>0$ such that,
\begin{equation}
  \label{control_Hlmi}
\forall l_\nu,m_\nu\in \bZ^{n_\nu}\ ,\qquad
\|H_{l_\nu,m_\nu}^{(\nu)}\|_{\cS(\bR^{n_\nu}), M}
\leq C_{M} (1+|l_\nu|+|m_\nu|)^{2M}
\ .  
\end{equation}
\item[\rm(b)] For any function $F\in \cS(\bR^n)$ 
there exist coefficients $c_{l,m}\in \bC$, 
$l=(l_1,l_2)$, $m=(m_1,m_2) \in \bZ^n=\bZ^{n_1}\times\bZ^{n_2}$,
such that
\begin{equation}
  \label{dec_F_sum_lm}
F=\sum_{l,m\in \bZ^n} c_{l,m} H_{l_1,m_1}^{(1)}\otimes H_{l_2,m_2}^{(2)}
\ .  
\end{equation}
\item[\rm(c)] For any $M\in\bN$, there is a constant $C_M$, independent of $F$,
such that the coefficients $c_{l,m}$ satisfy:
\begin{equation}
\label{control_clm}
\forall l,m\in \bZ^n\qquad
|c_{l,m}|\leq C_M \|F\|_{\cS(\bR^n),M} (1+ |l|+|m|)^{-M}
\ .
\end{equation}

\end{enumerate}
\end{lemma}

\begin{proof}[Proof of Lemma \ref{lemma_product}]
The inequalities in \eqref{control_Hlmi} are trivially satisfied.
For $\nu=1,2$,
we consider a partition of $\bR^{n_1}$ with the cubes $l_\nu+[-\frac12,\frac12]^{n_\nu}$, $l_\nu\in \bZ_{d_\nu}$, and the partition of unity obtained from a smooth function $\phi_\nu$ on $\bR^{n_\nu}$
supported in $[-\frac 34, \frac34]^{n_\nu}$ such that
$$
0\leq \phi_\nu\leq 1
\qquad,\qquad
\phi_\nu=1 \;\mbox{on}\; [-\frac 14, \frac14]^{n_\nu}
\qquad\mbox{and}\qquad
\forall x_\nu\in \bR^{n_\nu}\quad \sum_{l_\nu\in \bZ^{n_\nu}}\phi_\nu(x_\nu+l_\nu)=1
\ .
$$
In particular, $\phi_\nu=\phi_\nu\psi_\nu$.

For each $l=(l_1,l_2)\in \bZ^{n_1}\times\bZ^{n_2}$,
the function $x=(x_1,x_2)\mapsto F(x) \phi_1(x_1+l_1)\phi_2(x_2+l_2)$
is smooth and supported in $-l+[-\frac 34, \frac34]^n$. 
If $\sum_{m\in \bZ^n} c_{l,m} e^{ i x.m}$ is the Fourier series of its $2\pi$-periodic extension in each variable, it is easy to see that the coefficients $c_{l,m}$ satisfy \eqref{control_clm} for any $M\in \bN$.
We can write
\begin{eqnarray*}
F(x)
&=&
\sum_{(l_1,l_2)\in \bZ^{n_1}\times \bZ^{n_2}}
F(x)\phi_1(x_1+l_1)\phi_2(x_2+l_2) \\
&=&
\sum_{(l_1,l_2)\in \bZ^{n_1}\times \bZ^{n_2}}
F(x)\phi_1(x_1+l_1)\phi_2(x_2+l_2) 
\psi_1(x_1+l_1)\psi_2(x_2+l_2) 
\\
&=&
\sum_{l,m\in \bZ^n}
c_{l,m} e^{ix\cdot m}
\psi_1(x_1+l_1)\psi_2(x_2+l_2) 
\ ,
\end{eqnarray*}
which gives \eqref{dec_F_sum_lm}.
\end{proof}

\begin{corollary}\label{Kinvar}
For $\nu=1,2$, let $K_\nu$ be a compact subgroup of ${\rm GL}_{n_\nu}(\bR)$. For $\nu=1,2$, $l_\nu,m_\nu\in\bZ^{n_\nu}$, there exist $K_\nu$-invariant smooth functions $\tilde H^{(\nu)}_{l_\nu,m_\nu}$ on $\bR^{n_\nu}$ such that the conclusions of Lemma \ref{lemma_product} hold, with $\tilde H^{(\nu)}_{l_\nu,m_\nu}$ in place of $H^{(\nu)}_{l_\nu,m_\nu}$, for $F$ in $\cS(\bR^n)$ and $K_1\times K_2$-invariant.
\end{corollary}

\begin{proof} Just take as $\tilde H^{(\nu)}_{l_\nu,m_\nu}$ the $K_\nu$-average of $H^{(\nu)}_{l_\nu,m_\nu}$. The conclusion is quite obvious.
\end{proof}

We can now give the proof of Proposition \ref{product}.

\begin{proof}[Proof of Proposition \ref{product}]
We fix two families, $\cD^{(1)}=(D_1^{(1)},\ldots, D_{d_1}^{(1)})$ 
and $\cD^{(2)}=(D_1^{(2)},\ldots,D_{d_2}^{(2)})$, of generators
of $\bD(N_1)^{K_1}$ and $\bD(N_2)^{K_2}$ respectively.
We keep the same notation when the operators $D_j^{(\nu)}$ 
are applied to functions on $N=N_1\times N_2$ by differentiating in the $N_\nu$-variables. 

Then 
$\cD=(D_1^{(1)},\ldots, D_{d_1}^{(1)},D_1^{(2)},\ldots,D_{d_2}^{(2)})$ is a family
of generators of $\bD(N)^K$, $K=K_1\times K_2$,
$$
\Sigma_\cD=\Sigma_{\cD_1}\times\Sigma_{\cD_2}\subset \bR^{d_1}\times\bR^{d_2}\ ,
$$
and, if $\cG_1$, $\cG_2$, $\cG$ are the corresponding Gelfand transforms, then 
$$
\cG(F_1\otimes F_2)=(\cG_1 F_1)\otimes(\cG_2F_2)\ .
$$

Identifying $N_1$ and $N_2$ with their Lie algebras,
we consider the $K_i$-invariant functions $\tilde H_{l_\nu,m_\nu}^{(\nu)}$, $\nu=1,2$ and $l_\nu,m_\nu\in \bZ^{n_\nu}$, 
satisfying the properties of Corollary \ref{Kinvar}.

Given $F\in \cS(N)^K$, we decompose it as
$$
F=\sum_{l,m\in \bZ^n} c_{l,m} \tilde H_{l_1,m_1}^{(1)}\otimes \tilde H_{l_2,m_2}^{(2)}
\ ,
$$
with coefficients $c_{l,m}$ satisfying \eqref{control_clm}.
Then
$$
\cG F=\sum_{l,m\in \bZ^n} c_{l,m}\cG_1 \tilde H_{l_1,m_1}^{(1)}\otimes \cG_2\tilde H_{l_2,m_2}^{(2)}
\ .
$$

Since we are assuming that each $(N_\nu,K_\nu)$ satisfies property (S), given any  $M\in \bN$, there are functions $h_{l_\nu,m_\nu}^{(\nu,M)}\in \cS(\bR^{d_\nu})$, for  $\nu=1,2$ and $l_\nu,m_\nu\in \bZ^{n_\nu}$, and an integer $A_M$ such that
\begin{enumerate}
\item[(i)] $h_{l_\nu,m_\nu}^{(\nu,M)}$ coincides with $\cG_\nu \tilde H_{l_\nu,m_\nu}^{(\nu)}$ on $\Sigma_{\cD_\nu}$,
\item[(ii)] $\|h_{l_\nu,m_\nu}^{(\nu,M)}\|_{\cS(\bR^{d_\nu}),M}\le C_M\|\tilde H_{l_\nu,m_\nu}^{(\nu)}\|_{\cS(N_\nu),A_M}\le C_M(1+|l_\nu|+|m_\nu|)^{2A_M}$.
\end{enumerate}

If we set
$$
h_{l,m}^{(M)}=h_{l_1,m_1}^{(1,M)}\otimes h_{l_2,m_2}^{(2,M)}
\in \bS(\bR^d)\ ,
$$
where $d=d_1+d_2$, we have, for any $l,m\in \bZ^n$, 
\begin{equation}
\label{control_hlmM}
\|h_{l,m}^{(M)}\|_{\cS(\bR^{d}),M}\leq C_M (1+|l|+|m|)^{2A_M}
\ .
\end{equation}

Combining the rapid decay of the coefficients with the polynomial growth \eqref{control_hlmM} of the $M$-th Schwartz norm of the $h_{l,m}^{(M)}$, we obtain that
\begin{equation}\label{B_M-estimate}
\sum_{l,m\in \bZ^n}|c_{l,m}|\|h_{l,m}^{(M)}\|_{\cS(\bR^{d}),M}\le C_M\|F\|_{\cS(N),B_M} \ ,
\end{equation}
with $B_M=2A_M+n+1$.

Given $M_0\in\bN$, we want to construct a Schwartz extension $f^{(M_0)}$ of $\cG F$ whose $M_0$-th Schwartz norm is controlled by a constant, independent of $F$, times $\|F\|_{\cS(N),B_{M_0}}$.

By \eqref{B_M-estimate}, for every $M> M_0$, there exists $a_M=a_{M,M_0,F}\in \bN$ such that
\begin{equation}
\label{series_2-M}
\sum_{\tiny\begin{array}{c}
l,m\in \bZ^n
\\
|l|+|m|\geq a_M
\end{array}
 }|c_{l,m}|\|h_{l,m}^{(M)}\|_{\cS(\bR^{d}),M}
\leq 2^{-M} \| F\|_{\cS(N), B_{M_0}}\ .
\end{equation}

The $a_M$ can be inductively chosen to be non-decreasing.
Then we define $f^{(M_0)}$ as
$$
f^{(M_0)}=
\sum_{\tiny\begin{array}{c}
l,m\in \bZ^n
\\
|l|+|m|< a_{M_0+1}
\end{array}}
 c_{l,m} h_{l,m}^{(M_0)}
 +
 \sum_{M=M_0+1}^\infty \!\!\!\!\!\!\!\!\!
 \sum_{\tiny\begin{array}{c}
l,m\in \bZ^n
\\
a_M\leq |l|+|m|< a_{M+1}
\end{array}}\!\!\!\!\!\!\!\!\!
c_{l,m} h_{l,m}^{(M)}
\ .
$$

Clearly, the series converges on $\Sigma_\cD$ to $\cG F$.
To show that it defines a Schwartz function on all of $\bR^d$, notice that, for every $M_1\in \bN$ with $M_1> M_0$ and every $M\ge M_1$, we have by \eqref{series_2-M} that
$$
\sum_{\tiny\begin{array}{c}
l,m\in \bZ^n
\\
a_M\leq |l|+|m|< a_{M+1}
\end{array} }\!\!\!\!\!\!\!\!\!
|c_{l,m}|\|h_{l,m}^{(M)}\|_{\cS(\bR^{d}),M_1}
\leq \!\!\!\!\!\!\!\!\!
\sum_{\tiny\begin{array}{c}
l,m\in \bZ^n
\\
a_M\leq |l|+|m|
\end{array} }\!\!\!\!\!\!\!\!\!
|c_{l,m}|\|h_{l,m}^{(M)}\|_{\cS(\bR^{d}),M}
\leq 2^{-M} \| F\|_{\cS(N), B_{M_0}}
\ .
$$

This implies that $\|f\|_{\cS(\bR^{d}), M_1}$ is finite. Moreover, combining \eqref{B_M-estimate} and \eqref{series_2-M} together, we have that
$$
\|f^{(M)}\|_{\cS(\bR^d),M_0}\le (C_{M_0}+2^{-M_0})\|F\|_{\cS(N),B_{M_0}}\ ,
$$
as required.
\end{proof}

\subsection{Quotienting by a subspace of $\fw$}
\label{subsec_quotient}

In this section,
we consider quotients in the derived algebra
and define various the objects attached with them
which will be essential in the study of 
quotient pairs later on.

\subsubsection{Push-forward of functions and differential operators}
We consider a general connected and simply connected nilpotent Lie group $N$ of step two, and a decomposition $\fn=\fv\oplus\fw$ of its Lie algebra with $\fw=[\fn,\fn]$.
We consider a non-trivial subspace $\fs$ of $\fw$.

We denote by $\fn'$ the quotient algebra $\fn/\fs$ and by $N'$ the corresponding quotient group. 
If $\fw'$ is a complement of $\fs$ in $\fw$, i.e. $\fw=\fw'\oplus \fs$, 
then $\fn'$ can be regarded as $\fv\oplus\fw'$ with Lie bracket
$$
[v,w]_{\fn'}={\rm proj}[v,w]\ ,
$$
where  ${\rm proj}$ denotes the  projection of $\fw$ onto $\fw'$ along $\fs$.
Note that $[\fn',\fn']_{\fn'}=\fw'$.

For any function $F\in \cS(N)$, we define the function $\cR^\fs F \in \cS(N')$ via
\begin{equation}\label{radon*}
\cR^\fs F(v,w')=\int_\fs F(v,w'+s)\,ds\ .
\end{equation}

For any  $F\in \cS(N)$ and $G\in L^\infty(N')$, we have the identity
$$
\int_{N'}\cR^\fs F(v,w')G(v,w')\,dv\,dw'=\int_N F(v,w)(G\circ{\rm proj})(v,w)\,dv\,dw\ ,
$$
i.e., $\cR^\fs$ is the formal adjoint of the lifting operator from functions on $N'$ to functions on $N$.
Accordingly, given $D\in\bD(N)$, we define $\cR^\fs D\in\bD(N')$ as the operator such that
\begin{equation}\label{opradon*}
\big((\cR^\fs D)G\big)\circ{\rm proj}=D(G\circ{\rm proj})\ , 
\quad G\in \cC^\infty(N').
\end{equation}
Using the formula for the symmetrisation given in 
 \eqref{modsym}, we can check easily that 
\begin{equation}\label{opradon*1}
 D=\la'_N(p)\in\bD(N)
 \ \Longrightarrow \
\cR^\fs D=\la'_{N'}(p_{|_{\fn'}}) \in\bD(N').
\end{equation}

Moreover, for $F,G\in C^\infty(N)$, $D,D_1,D_2\in\bD(N)$,
\begin{equation}\label{composition-radon}
\begin{aligned}
\cR^s(F*_NG)&=(\cR^\fs F)*_{N'}(\cR^\fs G),\\
 \cR^s(DF)&=(\cR^\fs D)(\cR^\fs F),\\
  \cR^s(D_1D_2)&=(\cR^\fs D_1)(\cR^\fs D_2).
  \end{aligned}
\end{equation}

\subsubsection{Action of a compact group}
Keeping  the notation above, we also assume that we are given a compact group $K$ acting by automorphisms on $N$.
Hence we can endow $\fn$ with  a $K$-invariant inner product
and we choose $\fv=\fw^\perp$ in $\fn$ and  $\fw'=\fs^\perp$ in $\fw$.
Then ${\rm proj}$ is the orthogonal projection of $\fw$ onto $\fw'$. 
 We denote by $K'$  the stabiliser of $\fs$ in $K$ and fix Lebesgue measures $dv$, $dw'$,  $ds$ on $\fv$, $\fw'$,  $\fs$, respectively.

One checks readily that if a Schwartz $F$ is $K$-invariant, then $\cR^\fs F$ is also Schwartz and $K'$-invariant, i.e.,
$$
F\in\cS(N)^K\ \Longrightarrow \ \cR^\fs F\in \cS(N')^{K'}.
$$
For differential operators, if $D\in \bD(N)^K$ then $\cR^\fs D\in \bD(N')^{K'}$.

\subsubsection{Case of a nilpotent Gelfand pair}
We continue with the notation above.
Notice that if  $(N,K)$ is a nilpotent Gelfand pair
and if  $K'$ is a subgroup of $K$ which stabilises $\fs$, 
then $(N',K')$ is not necessarily a nilpotent Gelfand pair. For future reference (Section \ref{sec_quotient_pairs}), we state the following sufficient condition.

\begin{lemma}\label{N'K'gelfand}
Assume that $(N,K)$ is a nilpotent Gelfand pair and, for generic $\zeta'\in\fw'$, the stabiliser $K_{\zeta'}$ of $\zeta'$ in $K$ is contained in $K'$. 
Then $(N',K')$ is a nilpotent Gelfand pair.
\end{lemma}

\begin{proof}
For $\zeta'\in\fw'$, the representations $\pi_{\zeta',\omega}$  of $N$ defined in Lemma \ref{representations} factor to $N'$ giving all its irreducible unitary representations. According to Proposition \ref{carcano}, $(N',K')$ is a nilpotent Gelfand pair if and only if, for generic $\zeta',\omega$, the representation space decomposes without multiplicities under the action of $K'_{\zeta'}$. Under the present hypotheses, this condition is satisfied because $K'_{\zeta'}=K_{\zeta'}$. 
\end{proof}

In the case that both $(N,K)$ and $(N',K')$ are nilpotent Gelfand pairs, 
 we can define a map 
(denoted $\Lambda^\fs$ below)
between the  Gelfand spectra.

\begin{proposition}\label{propradon*}
Let $(N,K)$ be a nilpotent Gelfand pair.
We consider a non-trivial subspace $\fs$ of the derived algebra $\fw$ of $\fn$
and a subgroup $K'$ of $K$ which stabilises $\fs$.
Suppose that $(N',K')$ is  a nilpotent Gelfand pair, where $\fn'=\fn/\fs$.

If  $\ph'$ is a bounded spherical function on $N'$,  then 
the function $\Lambda^\fs \ph'$ defined via
$$
\Lambda^\fs\ph'(v,w)=\int_K(\ph'\circ{\rm proj})(kv,kw)\,dk\ ,
\quad (v,w)\in N,
$$
is a bounded spherical function on~$N$. 
This defines a map $\Lambda^\fs:\Sigma(N',K')\to\Sigma(N,K)$
which is continuous.

Denoting by $\cG$ and $\cG'$ the spherical transform of $(N,K)$ and $(N',K')$, 
we have
$$
\forall F\in \cS(N)^K, \quad \forall\ph'\in \Sigma(N',K')\qquad
\cG'(\cR^\fs F)(\ph') = \cG (F)(\Lambda^\fs \ph').
$$

\end{proposition}
 
\begin{proof}
We keep the notation of the statement.
One checks easily that $\Lambda^\fs\ph'$ is a bounded smooth $K$-invariant function on $N$ which is equal to 1 at 0. 
It remains to see that it is an eigenfunction for every $D\in\bD(N)^K$
and, for this,
we compute easily using \eqref{opradon*}:
$$
\begin{aligned}
D(\Lambda^\fs\ph')(v,w)&=\int_KD(\ph'\circ{\rm proj})(kv,kw)\,dk\\
&=\int_K\big((\cR^\fs D)\ph'\big)\circ{\rm proj}(kv,kw)\,dk\\
&=\xi \,\Lambda^\fs\ph'(v,w)\ ,
\end{aligned}
$$
where $\xi=\xi(\cR^\fs D,\ph')$. 
This proves that $\Lambda^\fs\ph'$ is a bounded spherical function for $(N,K)$. Hence the map $\Lambda^\fs$ is well defined.
One checks easily the continuity of $\Lambda^\fs$ for the compact-open topology.

For any $F\in \cS(N)^K$
and $\ph'\in \Sigma(N',K')$, we have
\begin{eqnarray*}
\cG (F)(\Lambda^\fs \ph')
&=&
\int_{\fv\times\fw} F(v,w) (\ph'\circ{\rm proj})(v,w) dv dw
\\
&=&
\int_{\fv\times\fw'} \cR^\fs F(v,w') \ph'(v,w') dv dw'
\\
&=&
\cG ' (\cR^\fs F)(\ph').
\end{eqnarray*}
This shows the last property of the statement.
\end{proof}

Keeping the notation of Proposition \ref{propradon*}, 
we see that we have also obtained the following relation between  tuples of
 eigenvalues associated with $\ph'$ and $\Lambda^\fs \ph'$:
\begin{equation}\label{radon-eigenvalues}
\forall D\in\bD(N)^K\qquad
\xi(D,\Lambda^\fs\ph')=\xi(\cR^\fs D,\ph')\ .
\end{equation}
In particular, this implies that the map $\Lambda^\fs$ is smooth in the following sense:

\begin{proposition}\label{Lambda-smooth}
In the setting of Proposition \ref{propradon*}, 
let $\cD$, resp. $\cD'$, be a $d$-tuple, resp. a $d'$-tuple, of operators generating the algebra $\bD(N)^K$, resp. $\bD(N')^{K'}$.
Then $\Lambda^\fs$, regarded as a map from $\Sigma'_{\cD'}$ to $\Sigma_{\cD}$, is the restriction of a  polynomial map from~$\bR^{d'}$ to~$\bR^{d}$.
\end{proposition}

\begin{proof}
Let $\cD=(D_1,\dots,D_d)$ 
and $\cD'=(D'_1,\dots,D'_{d'})$. For each $j=1,\dots,d$, $\cR^\fs D_j\in\bD(N')^{K'}$. Hence there exists a polynomial $p_j$ in $d'$ variables such that 
$\cR^\fs D_j=p_j(\cD')$. By \eqref{radon-eigenvalues}, given a $K'$-spherical function $\ph'$ on $N'$ with $d'$-tuple of  eigenvalues $\xi'$, we have
$$
\xi(D_j,\Lambda^\fs\ph')=p_j(\xi')\ .\qquad\qedhere
$$
\end{proof}

Parametrising the bounded spherical function
as  in \eqref{trace} and keeping the notation of Lemma~\ref{representations}, 
we can describe the image of $\Lambda^\fs$ in the following way.
\begin{proposition}
\label{lem_propradon*}
In the setting of Proposition \ref{propradon*}, let $\ph'=\ph'_{\zeta,\omega,\nu}$ be a bounded spherical function on $N'$
as in \eqref{trace}, where $\zeta\in\fw'$, $\omega\in\fr_\zeta$ and $V'(\nu)$  
an  irreducible $K'_{\zeta,\omega}$-invariant subspace of the representation space $\cH_\zeta$. If $V(\mu)$ is the  irreducible $K_{\zeta,\omega}$-invariant subspace containing $V'(\nu)$, we have:
$$
\Lambda^\fs \ph'_{\zeta,\omega,\nu} = \ph_{\zeta,\omega,\mu}.
$$

In particular, the image of $\Lambda^\fs$ consists of those $\ph\in\Sigma$ which can be expressed as $\ph_{\zeta,\omega,\mu}$ for some $\zeta\in \fw'$.
\end{proposition}

\begin{proof}
Let $\pi'_{\zeta,\omega}$ be the representation of $N'$ associated with 
$(\zeta,\omega)$ as  in Lemma \eqref{representations}.
As $\zeta\in \fw'=\fs^\perp$, 
one checks easily that $\omega\in \fv$ is in the $\zeta$-radical for $\fn$ and $\fn'$. Hence we can also consider the unitary irreducible representation 
$\pi_{\zeta,\omega}$ of $N$.
Lemma \eqref{representations} implies  easily that
$\pi'_{\zeta,\omega}=(\pi_{\zeta,\omega})_{|N'}$
and furthermore that 
$$
\pi_{\zeta,\omega} = \pi'_{\zeta,\omega}\circ {\rm proj}.
$$

By \eqref{trace},  
the spherical function $\ph'_{\zeta,\omega,\nu}\in\Sigma (N',K')$ is given via
$$
\ph'_{\zeta,\omega,\nu}(v,w)
=
\frac1{\dim V'(\nu)}\int_{K'}\tr\big(\pi'_{\zeta,\omega}(hv,hw)_{|_{V'(\nu)}}\big)\,dh\ , 
\quad (v,w)\in N'.
$$
Hence, for any $(v,w)\in N$, we have
$$
\Lambda^\fs \ph'_{\zeta,\omega,\nu}(v,w)
=
\frac1{\dim V'(\nu)}
\int_{K}
\tr\big(\pi_{\zeta,\omega}(kv,kw)_{|_{V'(\nu)}}\big)\,dk\ .
$$

As $K'_{\zeta,\omega}\subset K_{\zeta,\omega}$, 
the decomposition $\cH_\zeta=\sum_{\nu\in\fX'_{\zeta,\omega}}V'(\nu)$
is a refinement of $\cH_\zeta=\sum_{\mu\in\fX_{\zeta,\omega}}V(\mu)$.
Since both decompositions are multiplicity-free,
each $V(\mu)$ is a finite union of $V'(\nu)$.
We observe that, by Schur's lemma, for every $\mu\in\fX_{\zeta,\omega}$ and every unit element $e\in V(\mu)$, we have 
$$
\int_K\lan\pi_{\zeta,\omega}(kv,kw)e,e\ran\,dk=\frac1{\dim V(\mu)}\int_K\tr\big(\pi_{\zeta,\omega}(kv,kw)_{|_{V(\mu)}}\big)\,dk\ ,
\quad (v,w)\in N.
$$
Consequently, 
when  $V'(\nu)\subset V(\mu)$, we have:
$$
\frac1{\dim V'(\nu)}
\int_{K}
\tr\big(\pi_{\zeta,\omega}(kv,kw)_{|_{V'(\nu)}}\big)\,dk
=
\frac1{\dim V(\mu)}
\int_{K}
\tr\big(\pi_{\zeta,\omega}(kv,kw)_{|_{V(\mu)}}\big)\,dk,
$$
and thus $\Lambda^\fs \ph'_{\zeta,\omega,\nu}
=\ph_{\zeta,\omega,\mu}$.
\end{proof}

Propositions \ref{propradon*}-\ref{lem_propradon*} will be an essential tool in Section \ref{subsec_quotient_pairs} to understand the local identifications between the spectrum of a given nilpotent Gelfand pair $(N,K)$ and those of its quotient pairs.
At the moment, we use these results to conclude the present section with the proof of Proposition \ref{central}.

\subsection{Proof of Proposition \ref{central}}
\label{subsec_pf_prop_central}\

We keep the same notation of Section \ref{subsec_quotient} for
 $N, \fn=\fv\oplus \fw, \fs \subset \fw, 
N',\cR^\fs,K$.
We suppose furthermore that $(N,K)$ is a nilpotent Gelfand pair and 
that the subspace $\fs$  is invariant under $K$.
The group $K'$ stabilising~$\fs$ is then $K$ itself.
It is proved in \cite{V2} that $(N',K)$ is a nilpotent Gelfand pair.

Simplifying the proof of Proposition \ref{propradon*},
one  easily checks that, under these assumptions, 
 if $\ph'$ is spherical and bounded on $N'$,  
 then  $\ph'\circ {\rm proj} $ is also spherical and bounded on $N$.
 Moreover, the map $\Lambda^\fs$ in Proposition~\ref{propradon*} reduces to composition with ${\rm proj}$, that is, 
 $$
 \Lambda^\fs \ph' = \ph' \circ {\rm proj}.
 $$
In particular, $\Lambda^\fs$ is injective and admits a simple realisation once we  choose suitable generators for $\bD(N)^K$ and $\bD(N')^K$
 in the following way.
We fix a  Hilbert basis $(\rho_1,\dots,\rho_{d'})$ of $K$-invariants on $\fv\oplus\fw'$, and complete it into a  Hilbert basis $(\rho_1,\dots,\rho_{d'},\rho_{d'+1},\dots,\rho_d)$ of $K$-invariants on $\fn$, where each of the added polynomials $\rho_{d'+1}(v,w'+s),\dots,\rho_d(v,w'+s)$ vanishes for $s=0$.
Applying the symmetrisation  \eqref{modsym} on $N$ and $N'$ respectively, we obtain the generating systems of differential operators 
$$
\cD=(D_1,\dots,D_d)\ ,\qquad \cD'=(D'_1,\dots,D'_{d'})\ ,
$$
on $N$ and $N'$ respectively. 
The equality in \eqref{opradon*1} implies that
$$
\cR^\fs D_j=
\begin{cases}
D'_j&\text{ if } j=1,\dots,d'\ ,\\
0&\text{ if }  j=d'\!\!+\!1,\dots,d\ .
\end{cases}
$$

The link between the eigenvalues associated with $\ph'$ and $\Lambda^\fs \ph'$
in \eqref{radon-eigenvalues} then yields
$\xi (\Lambda^\fs\ph')  = \big(\xi (\ph'),0\big)$.
This shows that, 
if the Gelfand spectra are realised with $\cD$ and $\cD'$,
the map $\Lambda^\fs$ becomes the injection
 $\xi'\longmapsto (\xi',0)$
 from $\Sigma_{\cD'} \subset \bR^{d'}$ into $\Sigma_\cD\subset \bR^{d}$. This allows us to consider $\Sigma_{\cD'}\times \{0\}$ as a subset of $\Sigma_\cD$.
 
 Let $\cG$ and $\cG'$ be the spherical transforms
of $(N,K)$ and $(N',K')$ respectively.
We fix a smooth and compactly supported function $\psi$ on $\fs$ with integral 1.
Hence for any $F\in \cS(N')^K$,
$(F\otimes\psi)(v,w'+s)= F(v,w')\psi(s)$ defines a Schwartz function $F\otimes\psi \in \cS(N)^K$
with 
$\cR^\fs (F\otimes \psi)=F$, 
and, by Proposition \ref{propradon*}, 
we have $\cG' F(\ph') 
=\cG (F\otimes \psi) (\Lambda^\fs\ph')$ for any
$\ph'\in \Sigma(N',K)$.
Realising the Gelfand spectra as $\Sigma_\cD$ and $\Sigma_{\cD'}$, 
we have obtained 
\begin{equation}
\cG'F (\xi')
=
\cG (F\otimes \psi) (\xi',0)
\quad\mbox{for any}\  \xi' \in \Sigma_{\cD'}.
\label{Gelfand_transforms}
\end{equation}

One checks readily that the map $F\mapsto F\otimes\psi$ is continuous  and linear from $\cS(N')^K$ to $\cS(N)^K$. 

We now assume that $(N,K)$ satisfies property (S).
This means that 
$\cG:\cS(N)^K\rightarrow \cS(\Sigma_\cD)$ 
is an isomorphism of Fr\'echet spaces. 
Consequently,
for any $F\in \cS(N')^K$, 
$\cG (F\otimes \psi)$ extends to a Schwartz function $f \in \cS(\bR^d)$.
The equality in \eqref{Gelfand_transforms} implies that 
$\cG' F$ extends to the function $\xi'\mapsto f(\xi',0) $
which is in $\cS(\bR^{d'})$.
Together with \eqref{inclusion}, this shows that $(N',K)$ satisfies property (S), and this is the desired conclusion.

\section{Quotient pairs, slices and radialisation}
\label{sec_quotient_pairs}
In this section, we define the notion of quotient pairs which appears in the formulation of property (S). In order to prove that they are nilpotent Gelfand pairs, we must appeal to the slice theorem for  actions of compact Lie groups on vector spaces. This is done in Sections \ref{sec-slices} and \ref{subsec_quotient_pairs}, where the related notion of radialisation of a function on a slice is also presented. In Section \ref{special} we fix convenient Hilbert bases which will be used in  Section \ref{subsec_rad_quotient} to establish relations between the Gelfand spectrum of $(N,K)$ and that of a quotient pair.
%

We fix a nilpotent Gelfand pair $(N,K)$.
Given a point $t\in\fw$, 
we define the two following subspaces of $\fw$:
$$
\ft_t:=\fk\cdot t 
\qquad\mbox{and}\qquad
\fw_t:=(\fk \cdot t)^\perp .
$$
Here $\fk$ is the Lie algebra of $K$ thus  $\ft_t$ is the tangent space in $t$ to the $K$-orbit  $K\cdot t$ in $\fw$.
As  the action of~$\fk$ is skew adjoint,  
we have $t\in \fw_t$. With respect to the decomposition $\fw=\fw_0\oplus\check\fw$ in \eqref{z_0}, we also have  $\check \fw\subset \fw_t$.

In order to avoid the trivial situation $\ft_t=\{0\}, \fw_t=\fw$, we also assume  $t\not\in\check\fw$. So, for $t\in\fw\setminus\check\fw$, we consider the quotient algebra 
$$
\fn_t:=\fn/\ft_t,
$$
denoting the canonical projection by ${\rm proj}_t$. Notice however that the decomposition 
$$
\fw=\ft_t\oplus\fw_t
$$
only depends on the $\fw_0$-component of $t$.

As in Section~\ref{subsec_quotient}, we regard $\fn_t$ as $\fv\oplus\fw_t$, with Lie bracket $[v,v']_{\fn_t}={\rm proj}_t[v,v']$. By $N_t$ we denote the quotient group $N/\exp\ft_t$.

We observe that the subspaces $\ft_t$ and $\fw_t$ of $\fw$ are invariant under the action of the stabiliser $K_t$ of $t$ in $K$.
Hence,  passing to the quotient, we obtain an action of $K_t$ on $N_t$. We call $(N_t,K_t)$ a {\it quotient pair} of $(N,K)$.

\subsection{Generalities on slices and radialisation}
\label{sec-slices}

We start this section by recalling some known facts about action of a compact groups on a vector space and the notion of slices.

A construction of a {\it slice} for a compact group action goes back to 
Gleason \cite{Glea}. We will need the ``linear" version of the slice theorem.

\begin{theorem}\label{Slice-Thm}
Let $W$ be a Euclidean vector space 
and let $K$ be a compact real Lie group $K$ acting orthogonally on $W$.
For any $x\in W$, we denote by
$K_x$ the stabiliser of $x$ in $K$
and by
$(\fk\cdot x)^\perp$  the normal space to the orbit $Kx$ at $x$.
There is an open and $K_x$-invariant (Euclidean) neighbourhood  $S_x$ of $0$ in $(\fk\cdot x)^\perp$
such that the $K$-equivariant map
$$
\sigma:K\times_{K_x} S_x \longrightarrow W,
$$
given by $\sigma(k,y)= k(x+y)$, 
is
a diffeomorphism of $K\times_{K_x}S_x$ onto the open neighbourhood $K(x+S_x)$ of~$Kx$.
\end{theorem} 

We call $S_x$ a {\it slice} at $x$. The notation $K\times_{K_x} S_x$ stands for the quotient of $K\times S_x$ modulo  the action of~$K_x$, i.e., $(kk',x)$ is equivalent to $(k,k'x)$ for $k'\in K_x$.


The proof of Theorem~\ref{Slice-Thm} can be found, for instance, in \cite[Ch.~2,\,Section~5]{bre}, 
in particular, see Corollary~5.2 therein. 
The theorem has the following almost immediate consequence (for part (i), see e.g. \cite[Ch.~2,\,Sections~4,~5]{bre}).
  
\begin{corollary}\label{Slice-cor}
We keep the notation of Theorem \ref{Slice-Thm}. 
\begin{enumerate}
\item[\rm(i)] For every $y\in x+ S_x$ we have the inclusion $K_y\subset K_x$, more explicitly  
$K_y=(K_x)_y$.
\item[\rm(ii)]  Two points in $x+S_x$ are conjugate under $K$ if and only if they are conjugate under~$K_x$.
\item[\rm(iii)] Suppose that $f$ is a $K_x$-invariant smooth function 
on $x+S_x$.  Then $f$
extends in a unique way to a~smooth $K$-invariant function $f^{\rm rad}$ 
on $K(x+S_x)$.
\end{enumerate}
\end{corollary}

We call $f^{\rm rad}$ the {\it radialisation of} $f$. 

\begin{remark}
{\rm This notion of radialisation extends the one used in \cite{FR} to  general pairs, including the rank-one pairs considered in \cite{FRY1}. We notice that  the  arguments in the rest of this paper will not rely on the results of \cite{FRY1}, which in fact are being given a different proof.}
\end{remark}

We also have the following consequence in terms of Hilbert bases. Let ${\boldsymbol\rho}=(\rho_1,\ldots,\rho_{d_\rho})$ be a bi-homogeneous real Hilbert basis for the orthogonal action of a compact real Lie group $K$ on a Euclidean vector space $W$.
We fix a point $x\in W$
and we consider a  bi-homogeneous real Hilbert basis ${\boldsymbol\tau}=(\tau_1,\ldots,\tau_{d_\tau})$ for the action of the stabiliser $K_x$ of $x$ in $K$ on the normal space
$(\fk\cdot x)^\perp$   to the orbit $Kx$ at~$x$.
By the properties of Hilbert bases, 
 there exists a polynomial map $P:\bR^{d_\tau}\to \bR^{d_\rho}$ such that 
${\boldsymbol\rho}_{|(\fk\cdot x)^\perp} = P\circ {\boldsymbol\tau}$ on~$(\fk\cdot x)^\perp$.

The next statement says that $P$ admits a smooth right-inverse on a slice at $x$. In order to formulate it precisely, we need to introduce the space
\begin{equation}\label{slowly increasing}
\cO(W)=\big\{f:W\longrightarrow\bC: \forall\,\al\ \, \exists\, p_\al\in\cP(W) \text{ s.t. }|\de^\al f|\le p_\al\big\}.
\end{equation}

A scalar (or vector-valued) function is called a {\it slowly increasing smooth function} if it (or each of its components) is in $\cO(W)$.

\begin{corollary}
\label{cor_rad}
Let $S_x$ be an open $K_x$-invariant neighbourhood of 0 in $(\fk\cdot x)^\perp$ as in Theorem \ref{Slice-Thm}.
Then $x+S_x\subset (\fk\cdot x)^\perp$.
Let $U$ be a Euclidean neighbourhood of $x$
such that $U$ is $K_x$-invariant, relatively compact and strictly included in $x+S_x$.
There exists a  slowly increasing smooth
 map $\Psi:\bR^{d_\rho}\to \bR^{d_\tau}$ such that 
${\boldsymbol\tau} = \Psi\circ {\boldsymbol\rho}$ on $U$, i.e.,
$$
 \forall w\in U,\qquad
{\boldsymbol\tau}(w)=\Psi\big({\boldsymbol\rho}(w)\big).
$$

Hence $P\circ \Psi=\id$ on ${\boldsymbol\rho}(U)$.
\end{corollary}
\begin{proof}
As the action of $K$ is orthogonal, one checks easily that 
$x\perp \fk \cdot x$ thus $x+S_x\subset (\fk\cdot x)^\perp$.

Let $\chi\in C^\infty_c(x+S_x)$ be $K_x$-invariant and equal to 1 on  $U$.   For $j=1,\dots,d_\tau$, thanks to Corollary \ref{Slice-cor}, 
we may define  $u_j\in C^\infty(W)$ via
$$
u_j(w)=
\left\{\begin{array}{ll}
(\chi \tau_j)^{\rm rad}(w),
&\mbox{if}\ w\in K(x+S_x)\\
0 & \mbox{if}\ w\not\in K(x+S_x).
\end{array}\right.
$$
By  G. Schwarz's theorem \cite{Schw}, 
there exists a smooth function $\Psi_j\in C^\infty(\bR^{d_{\boldsymbol\rho}})$
such that 
$u_j=\Psi_j\circ {\boldsymbol\rho}$ on $W$.
Writing $u=(u_1,\ldots,u_{d_\tau})$ and $\Psi=(\Psi_1,\ldots,\Psi_{d_\tau})$, we see that on $U$, 
${\boldsymbol\tau}=u=\Psi\circ {\boldsymbol\rho}$.

Using the homogeneity properties of the two bases, it is possible to construct $\Psi$ homogeneous under dyadic scaling, hence with polynomial growth in all derivatives, cf. \cite{ADR2}.
\end{proof}

\subsection{The Gelfand spectrum of a quotient pair}
\label{subsec_quotient_pairs}

We first observe that the quotient pairs defined at the beginning of Section \ref{sec_quotient_pairs} are Gelfand.

\begin{corollary}
Assume that $(N,K)$  is a nilpotent Gelfand pair. Then, for every $t\in\fw\setminus\check\fw$, also $(N_t,K_t)$ is a nilpotent Gelfand pair.
\end{corollary}

\begin{proof}
By Corollary \ref{Slice-cor}, Part (i), we can apply Lemma \ref{N'K'gelfand} with $\fw'=\fw_t$.
\end{proof}

We will denote by $\Sigma$, resp. $\Sigma^t$, the Gelfand spectrum of $(N,K)$, resp.  of its quotient pair $(N_t,K_t)$. We~also write $\cR^t$ instead of $\cR^{\ft_t}$ and
$$
\Lambda^t:\Sigma^t\to \Sigma
$$
instead of $\Lambda^{\ft_t}$ for the continuous mapping defined 
in Proposition \ref{propradon*}.

We denote by $\cG_t$ the Gelfand spectrum of the pair $(N_t,K_t)$.
By Proposition \ref{propradon*}, 
we have, for any  $F\in\cS(N)^K$,
\begin{equation}\label{G-G_t}
\cG_t(\cR^t F)=(\cG F)\circ\Lambda^t\ .
\end{equation}

Also recall that, 
by Proposition \ref{Lambda-smooth}, 
$\Lambda^t$ is the restriction of a polynomial map when regarded as a map from $\Sigma^t_{\cD^t}$ to $\Sigma_\cD$, for given generating systems $\cD$ and $\cD^t$ in $\bD(N)^K$ and $\bD(N_t)^{K_t}$ respectively.

\subsection{Choice of  Hilbert bases on $\fn$ and $\fn_t$}
\label{special}

We continue with the setting of Section \ref{subsec_quotient_pairs}.
We choose Hilbert bases ${\boldsymbol\rho}$ and
${\boldsymbol\rho}^t$
of $(N,K)$ and $(N_t,K_t)$ 
satisfying (i)-(iv) of Section \ref{subsec_hilbert_base}.
Denote by ${\boldsymbol\rho}_{(k)}$, resp. ${\boldsymbol\rho}^t_{(k)}$ the invariants in ${\boldsymbol\rho}$, resp. ${\boldsymbol\rho}^t$, which have degree $k$ in the $\fv$-variables  (in the ordinary sense). Then ${\boldsymbol\rho}_{(0)}={\boldsymbol\rho}_\fw$ and ${\boldsymbol\rho}^t_{(0)}={\boldsymbol\rho}^t_{\fw_t}$.
 We~label the elements in ${\boldsymbol\rho}_{(k)}$  
  as $(\rho_{k,1},\dots,\rho_{k,d_{(k)}})$.
  Similarly 
${\boldsymbol\rho}^t_{(k)}=(\rho^t_{k,1},\dots,\rho^t_{k,d^t_{(k)}})$.

The basic properties of Hilbert bases implies that the restriction of an element of ${\boldsymbol\rho}$ to $\fn_t$
can be expressed as a polynomial in the elements of 
${\boldsymbol\rho}^t$. 
Necessarily, for polynomials only in $\fw$, the restrictions must be of the~form:
\begin{equation}\label{Q0j}
{\rho_{0,j}}_{|_{\fw_t}}=
Q_{0,j}\circ{\boldsymbol\rho}^t_{(0)}
\end{equation}
whereas for non-zero $\fv$-degrees $k\in \bN$, we have 
\begin{equation}\label{Qkj}
{\rho_{k,j}}_{|_{\fn_t}}=
\sum_{\ell=1}^{d_{(k)}^t}
(Q_{k,j,\ell}\circ{\boldsymbol\rho}^t_{(0)}) \
\rho^t_{k,\ell}
\ 
+\ 
R_{k,j} \circ({\boldsymbol\rho}^t_{(0)}, \ldots, 
{\boldsymbol\rho}^t_{(k-1)} )
\end{equation}
where the $Q$'s and the $R$'s are polynomials, with
$R_{k,j} ({\boldsymbol\rho}^t_{(0)}, \ldots, 
{\boldsymbol\rho}^t_{(k-1)} )$
  of degree $k$ in $\fv$. 
  
 In an analogous way, 
we split $\cD=\lambda'_N(\boldsymbol\rho)$ and the corresponding $d$-tuples of eigenvalues as
\begin{equation}
\label{eq_splitting_k}
\cD=(\cD_{(k)})_{k\in \bN},
\quad 
\xi=(\xi_{(k)})_{k\in \bN} \in \bR^d,
\end{equation}
with $\cD_{(k)}=\la'_N({\boldsymbol \rho}_{(k)})$, and similarly, with $\cD^t=\lambda'_{N_t}({\boldsymbol\rho}^t)$,
\begin{equation}
\label{eq_splitting_kt}
\cD^t=(\cD_{(k)}^t)_{k\in \bN},
\quad 
\xi^t=(\xi^t_{(k)})_{k\in \bN} \in \bR^{d_t}.
\end{equation}

We apply the symmetrisation $\lambda'_{N_t}$ to both sides of \eqref{Q0j} and \eqref{Qkj}.

By Lemma \ref{lem_sym}, for polynomials $p$ on $\fn$ which only depend on $w$, we have the identity $\lambda'(p)=p(i\inv \nabla_\fw)$, and similarly on $\fn_t$. It follows that
\begin{equation}\label{D0j}
\cR^t D_{0,j}=
Q_{0,j}(\cD_{(0)}^t).
\end{equation}

For  $\fv$-degrees $k\ne0$, we have 
\begin{equation}\label{Dkj}
\cR^t D_{k,j}=
\sum_{\ell=1}^{d_{(k)}^t}
Q_{k,j,\ell}(\cD^t_{(0)}) \
D^t_{k,\ell}
\ 
+\  
R'_{k,j} (\cD_{(0)}^t,\ldots,\cD^t_{(k-1)}).
\end{equation}
The polynomials $Q_{0,j}$ 
in \eqref{D0j} and $Q_{k,j,l}$ in \eqref{Dkj}
are the same as in \eqref{Q0j} and \eqref{Qkj} respectively. 
However the polynomials $R_{k,j}$ in \eqref{Qkj} have to be modified  to take into account the lower-order terms in Lemma \ref{lem_sym} produced by symmetrisation.
Degree considerations 
imply that $\lambda'\big(R_{k,j} ( {\boldsymbol\rho}^t_{(0)}, \ldots, 
{\boldsymbol\rho}^t_{(k-1)} )\big)$ is a polynomial $R'_{k,j}$ in $\cD^t_{(0)}, \ldots,\cD^t_{(k-1)}$ only.

As in Section \ref{subsubsec_sph+cG}, 
to each  spherical function $\ph^t\in \Sigma(N_t,K_t)$, 
we associate the $\cD^t$-eigenvalues $\xi^t=\xi^t(\ph^t) \in \bR^{d_t}$
which we split into $\xi^t=(\xi^t_{(k)})_{k\in \bN}$,
see \eqref{eq_splitting_kt}.
Applying $\ph^t$ to \eqref{D0j} and \eqref{Dkj}, 
we obtain by \eqref{radon-eigenvalues}:

By \eqref{radon-eigenvalues}, the relations \eqref{D0j} and \eqref{Dkj} extend to relations between tuples of eigenvalues of pairs of corresponding spherical functions under $\Lambda^t$. Precisely, given $\xi^t\in\Sigma^t_{\cD^t}$ as in \eqref{eq_splitting_kt}, the point $\xi=\Lambda^t\xi^t\in \Sigma_\cD$ is given by
\begin{equation}
\label{eq_xi2xit}
\xi_{0,j} = 
Q_{0,j}(\xi_{(0)}^t) 
\quad\mbox{and}\quad
\xi_{k,j}  = 
\sum_{\ell=1}^{d_{(k)}^t}
Q_{k,j,\ell}(\xi_{(0)}^t) \xi_{k,\ell} ^t
\ +\ 
R'_{k,j} (\xi_{(0)} ^t,\ldots, \xi_{(k)}^t).
\end{equation}

We have obtained the following property:
\begin{lemma}
\label{lem_xi2xit}
When  realising the spectra of $(N,K)$ and $(N_t,K_t)$ as  
 $\Sigma_\cD$ and $\Sigma^t_{\cD^t}$ respectively, 
the map 
$\Lambda^t$ is  given by
$\Lambda^t(\xi^t)= \xi$
where $\xi^t=(\xi^t_{(k)})_{k\in \bN}$, $\xi=(\xi_{(k)})_{k\in \bN}$
as in \eqref{eq_splitting_k}, \eqref{eq_splitting_kt}
and the components of each $\xi_{(k)}$ are given by~\eqref{eq_xi2xit}.
\end{lemma}

\subsection{Slices and radialisation applied to quotient pairs}
\label{subsec_rad_quotient}

We will apply the results of Sections \ref{subsec_quotient} and~\ref{sec-slices}  to prove that the map $\Lambda^t$ can be locally inverted by means of smooth functions on the set of spherical functions $\ph_{\zeta,\omega,\mu}$ with $\zeta$ close to $t$. 

We continue with the notation of  Sections \ref{subsec_quotient_pairs} and \ref{special}
regarding quotient pairs $(N_t,K_t)$
and the choice of~Hilbert bases 
${\boldsymbol\rho}$ and ${\boldsymbol\rho}^t$.
 For each $t\in\fw\setminus\check\fw$,
  let $S_t\subset\fw_t$ be a slice at $t$ in~$\fw$. 

The first step consists in expressing locally the elements of ${\boldsymbol\rho}^t$ as smooth functions of ${\boldsymbol\rho}$.
 We decompose $\xi\in \bR^d$ 
as $\xi=(\xi_{(k)})_{k\ge0}$,  $\xi^t\in \bR^{d_t}$ as $(\xi_{(k)}^t)_{k\ge0}$ as in Section \ref{subsec_quotient}, and express $\Phi(\xi)$ as $\big(\Phi_{(k)}(\xi)\big)_{k\ge0}$.

\begin{proposition}\label{Qinv}
Let $(N,K)$ be a nilpotent Gelfand pair, $t\in\fw\setminus\check\fw$ and $U$ be a Euclidean neighbourhood of~$t$,  $K_t$-invariant and relatively compact in $t+S_t$. 

 Then there exists a slowly increasing smooth function 
 $\Phi:\bR^{d} \to \bR^{d^t}$ such that 
 $$
 {\boldsymbol\rho}^t=\Phi\circ {\boldsymbol\rho}
 \quad\mbox{on}\  \fv\times U,
 $$
 and, denoting by $\Phi_{(k)}$ the component of $\Phi$ with values in $\bR^{d_{(k)}^t}$,
 \begin{enumerate}
 \item[\rm(1)] $\Phi_{(0)}$ only depends  on $\xi_{(0)}$,
  \item[\rm(2)]\label{2} for $k>0$,
$\Phi_{(k)}$  only depends  on $(\xi_{(0)},\ldots, \xi_{(k)})$
 and each of its component has the form
\begin{equation}\label{Phi-kj}
\Phi_{k,j}(\xi)=
\Phi_{k,j}(\xi_{(0)},\ldots, \xi_{(k)})=\sum_{\ell=1}^{d_{k_{(0)}}}\Phi_{k,j,\ell}(\xi_{(0)})\xi_{k,\ell}+\sum_\al\Psi_{k,j,\al}(\xi_{(0)})\xi_{(1)}^{\al_1}\xi_{(2)}^{\al_2}\cdots \xi_{(k-1)}^{\al_{(k-1)}},
\end{equation}
where $\Phi_{k,j,\ell},\Psi_{k,j,\al}\in \cO(\bR^{d_{(0)}})$
and the summation is extended to those $\al$ such that the polynomial ${\boldsymbol\rho}_{(1)}^{\al_1}{\boldsymbol\rho}_{(2)}^{\al_2}\cdots {\boldsymbol\rho}_{(k-1)}^{\al_{(k-1)}}$ has degree $k$ in $v$.
 \end{enumerate}
 \end{proposition}
 
The proof of Proposition \ref{Qinv}, given below, makes use of the following consequence of the G. Schwarz theorem \cite{Schw}, which can be shown by adapting the arguments in
\cite[Prop. 2.1]{FRY2}. We recall that the space $\cO(\fw)$ of slowly increasing smooth functions has been defined in~\eqref{slowly increasing}.

\begin{lemma}\label{tensor-invariants}
We have the identity
$$
\big(\cP(\fv)\otimes \cO(\fw)\big)^K=\cP(\fn)^K \cO(\fw)^K,
$$
i.e., every function $\psi\in\big(\cP(\fv)\otimes \cO(\fw)\big)^K$ can be decomposed as a finite sum
\begin{equation}\label{psi-invariant}
\psi(v,w)=\sum_\al u_\al({\boldsymbol\rho}_{(0)}){\boldsymbol\rho}_{(1)}^{\al_1}{\boldsymbol\rho}_{(2)}^{\al_2}\cdots{\boldsymbol\rho}_{(\bar k)}^{\al_{\bar k}},
\end{equation}
with $u_\al\in \cO(\fw)^K$.
\end{lemma}

\begin{proof}[Proof of Proposition \ref{Qinv}]
The existence of 
$\Phi_{(0)}$ follows directly from
Corollary \ref{cor_rad} applied to $W=\fw$ and the Hilbert bases ${\boldsymbol\rho}_{(0)}$, ${\boldsymbol\rho}_{(0)}^t$. 

Formula \eqref{Phi-kj} is proved with the following modification of the proof of  Corollary \ref{cor_rad}. 
For $k>0$ we regard $\fv\times S_t$ as a slice  at $(0,t)$ in $\fv\oplus\fw$.
Let $\chi\in C^\infty_c(t+S_t)$ be $K_t$-invariant and equal to 1 on  $U$. 
The~radialisation of $\chi\rho^t_{k,\ell}$
is a smooth $K$-invariant function on $\fv\times \fw$.

Since $(\chi\rho^t_{k,\ell})^{\rm rad}(v,gw)=\chi(w)\rho^t_{k,\ell}(g^{-1}v,w)$ for $w\in t+S_t$ and $g\in K$, $(\chi\rho^t_{k,\ell})^{\rm rad}$ is   a homogeneous polynomial of degree $k$ in the $\fv$-variable for every $w$, i.e., in $\big(\cP(\fv)\otimes\cO(\fw)\big)^K$. Then \eqref{Phi-kj} follows from Lemma \ref{tensor-invariants} and homogeneity considerations.
\end{proof}

The parametrisation of bounded spherical functions in \eqref{Sigma-triples}
together with the slice properties in Theorem \ref{Slice-Thm}
and Corollary \ref{Slice-cor}
 implies the `local bijectivity' of $\Lambda^t$ between  the subsets of the two spectra where the~parameter $\zeta$ can be taken
 in $t+S_t$.

\begin{proposition}\label{bijection}
 For each $t\in\fw\setminus\check\fw$,
  let $S_t\subset\fw_t$ be a slice at $t$ in $\fw$.
The map $\Lambda^t$ is a bijection from the~set
$$
A^t_t=\{\ph^t_{\zeta,\omega,\mu}:\zeta\in t+S_t, \omega\in \fr_\zeta, \mu \in \fX_{\zeta,\omega}\}\subset\Sigma^t
$$
onto the set
$$
A_t=\{\ph_{\zeta,\omega,\mu}:\zeta\in t+S_t, \omega\in \fr_\zeta, \mu \in \fX_{\zeta,\omega}\}\subset\Sigma. 
$$
\end{proposition}

\begin{proof}
For $\zeta\in t+ S_t$, $K_\zeta=(K_t)_\zeta$ by Corollary~\ref{Slice-cor}~(i). Consequently, $K_{(\zeta,\omega)}=(K_t)_{(\zeta,\omega)}$ for any $\omega\in \fr_\zeta$, 
so that the decomposition of $\cH_\zeta$ is the same.
The $K$-equivariance of spherical function, see \eqref{eq_ph_Kequivariance}, implies that 
the spherical function $\ph_{\zeta,\omega,\mu}\in \Sigma$ 
given  in \eqref{trace}
with $\zeta\in K\cdot (t+ S_t)$, i.e. $\zeta=k\zeta_t$ with $\zeta_t\in t+S_t\subset \fw_t $ satisfy
$$
\ph_{\zeta,\omega,\mu}
=
\ph_{k\zeta_t,\omega,\mu}
=
\ph_{\zeta_t,k^{-1}\omega,\mu}
=
\Lambda^t 
\ph_{\zeta_t,k^{-1}\omega,\mu}.
$$
This easily implies the statement.
\end{proof}

Note that Lemma \ref{lem_xi2xit} already described the map $\Lambda_t$, but in terms of the embeddings $\Sigma_\cD$, $\Sigma^t_{\cD^t}$ of the two spectra, whereas Proposition \ref{bijection} states the existence of the inverse map $\Lambda_t^{-1}:A_t\longrightarrow A_t^t$. 
In the next section we described 
$\Lambda_t^{-1}:A_t\longrightarrow A_t^t$
 in terms of the embeddings $\Sigma_\cD$, $\Sigma^t_{\cD^t}$.

\subsection{Extension of the symmetrisation}
\label{subsec_extlambda'}

The next step consists in transforming the relations between the two Hilbert bases ${\boldsymbol\rho},{\boldsymbol\rho}^t$ obtained in Proposition \ref{Qinv} into relations among the corresponding differential operators. In order to do so, we need to extend the notion of symmetrisations to functions that are not polynomials in $w$.

We continue with the notation of  Section \ref{subsec_quotient_pairs}
regarding quotient pairs $(N_t,K_t)$ of a given nilpotent pair $(N,K)$
and the choice of Hilbert bases 
${\boldsymbol\rho}$ and ${\boldsymbol\rho}^t$
and their symmetrisations $\cD=\la'_N({\boldsymbol\rho})$,
 $\cD_t=\la'_{N_t}({\boldsymbol\rho}^t)$.
 For each $t\in\fw\setminus\check\fw$,
 let $S_t\subset\fw_t$ be a slice at $t$ in $\fw$, and let $U$ be an open, relatively compact, $K_t$-invariant neighbourhood of~$t$  in $t+S_t$;
 we also consider the smooth function $\Phi:\bR^{d} \to \bR^{d^t}$ described in Proposition~\ref{Qinv}.
 
Observe first of all that, for $q\in\cP(\fw)$, the operator $\la'_N(q)=q(i\inv \nabla_\fw)$ can be expressed as a Fourier multiplier operator in the $w$-variables: 
$$
\la'_N(q)F=\cF_\fw^{-1}(q\cF_\fw F),
$$
where $\cF_\fw$ denotes the partial Fourier transform of $F$ in  $w$,
\begin{equation}\label{w-Fourier}
\cF_\fw F(v,w')=\int_\fw F(v,w)e^{-i \lan w,w'\ran}\,dw.
\end{equation}

In this form, we can replace the polynomial $q$ by any slowly increasing smooth function $u\in\cO(\fw)$  and set
\begin{equation}\label{lambda'u}
\lambda'_N(u)F=\cF_\fw^{-1}(u\cF_\fw F)=F*(\del_0\otimes\cF\inv u)=(\del_0\otimes\cF\inv u)*F.
\end{equation}

We are using the fact that $\cO(\fw)$ can be identified as the space of pointwise multipliers of $\cS(\fw)$ into itself~\cite{T}.

The operators $\la'_N(u)$ are central, which allows a linear extension of $\la'_N$ to $\cP( \fv)\otimes \cO(\fw)$:
if $\psi=\sum_jp_j\otimes u_j$ with $p_j\in\cP( \fv)$, $u_j\in\cO(\fw)$ and $D_j=\la'_N(p_j)$, then
\begin{equation}\label{pou}
\begin{aligned}
\la'_N(\psi)F&=\sum_j\la'_N(p_j)\la'_N(u_j)F\\
&=\sum_jD_j\big(F*(\del_0\otimes\cF\inv u_j)\big)=\sum_j(D_jF)*(\del_0\otimes\cF\inv u_j)\big).
\end{aligned}
\end{equation}
 
 Notice that
 \begin{itemize}
 \item formula \eqref{pou} also makes sense for $F\in\cS'(N)$;
 \item $\lambda'_N(\psi)$ is $K$-invariant if and only if $\psi$ is $K$-invariant.
 \item if $\psi(v,w)=p(v,w)u(w)$ with $p\in\cP(\fn)$ and $u\in\cO(\fw)$, then $\la'_N(\psi)=\la'_N(p)\la'_N(u)$;
 \item  if $\psi$ is as above and $\pi_{\zeta,\omega}$ is the representation of $N$ defined in Lemma \ref{representations}, then, for $F\in\cS(N)$,
$$
 \pi_{\zeta,\omega}\big(\la'_N(\psi)F\big)=u(\zeta)\pi_{\zeta,\omega}(F)d\pi_{\zeta,\omega}\big(\la'_N(p)\big).
$$
 \end{itemize}
 
 We then set
 \begin{equation}\label{pi(pou)}
 d\pi_{\zeta,\omega}\big(\la'_N(\psi)\big)=u(\zeta)d\pi_{\zeta,\omega}\big(\la'_N(p)\big).
 \end{equation}
 
%
%
%
%
%
%
%

 \begin{lemma}\label{lambda-spherical}
 Let $\ph_{\zeta,\omega,\mu}$ be the spherical function \eqref{trace}. Then, given $\psi=\sum_jp_j(v,w)u_j(w)$ in $\big(\cP(\fv)\otimes\cO(\fw)\big)^K$, with $p_j\in\cP(\fn)^K$, $u_j\in\cO(\fw)^K$, we have
 $$
 \la'_N(\psi)\ph_{\zeta,\omega,\mu}=\Big(\sum_j\xi\big(\la'_N(p_j),\ph_{\zeta,\omega,\mu}\big)u_j(\zeta)\Big)\ph_{\zeta,\omega,\mu}.
 $$
 \end{lemma}
 
 \begin{proof} We have
 $$
  \la'_N(\psi_{\zeta,\omega,\mu})\ph_{\zeta,\omega,\mu}=\sum_j\xi\big(\la'_N(p_j),\ph_{\zeta,\omega,\mu}\big)\la'_N(u_j)\ph_{\zeta,\omega,\mu}.
  $$

 By \eqref{trace} and the definition of $\pi_{\zeta,\omega}$ in Lemma \ref{representations},
 $$
 \begin{aligned}
 \ph_{\zeta,\omega,\mu}(v,w)&=\frac1{\dim V(\mu)}\int_K \tr\big(\pi_{\zeta,\omega}(kv,kw)_{|_{V(\mu)}}\big)\,dk\\
 &=\frac1{\dim V(\mu)}\int_K \tr\big(\pi_{\zeta,\omega}(kv,0)_{|_{V(\mu)}}\big)e^{i\lan k\inv\zeta,w\ran}\,dk.
 \end{aligned}
 $$

Hence $\cF_\fw\ph_{\zeta,\omega,\mu}$ is a measure\footnote{It can be formally written as
$$
\cF_\fw\ph_{\zeta,\omega,\mu}=\frac1{\dim V(\mu)}\int_K \tr\big(\pi_{\zeta,\omega}(kv,0)_{|_{V(\mu)}}\big)\del_{k\inv\zeta}(w)\,dk.
$$} supported on $\fv\times K\zeta$. 

Then, for each $j$, $u_j\cF_\fw\ph_{\zeta,\omega,\mu}=u_j(\zeta)\cF_\fw\ph_{\zeta,\omega,\mu}$ and therefore
$$
\la'_N(u_j)\ph_{\zeta,\omega,\mu}=u_j(\zeta)\ph_{\zeta,\omega,\mu}.\qquad\qedhere
$$
 \end{proof}
 
Hence we may consider $\la'_N(\Phi\circ {\boldsymbol\rho})$ component-wise
and the extension of the properties in Lemma \ref{lem_sym} yields:
\begin{enumerate}
\item For $k=0$, observing that $\Phi_{(0)}\circ {\boldsymbol\rho}_{(0)}\in\cO(\fw)$ by Corollary \ref{cor_rad}, we have
$$
\la'_N(\Phi_{(0)}\circ {\boldsymbol\rho})F= \la'_N(\Phi_{(0)}\circ {\boldsymbol\rho}_{(0)})F=\cF_\fw^{-1}\big((\Phi_{(0)}\circ {\boldsymbol\rho}_{(0)})\cF_\fw F\big).
$$
We can formally write
$$
\la'_N(\Phi_{(0)}\circ {\boldsymbol\rho})=(\Phi_{(0)}\circ {\boldsymbol\rho}_{(0)})(i\inv\nabla_{\fw})=\Phi_{(0)}(\cD_{(0)}).
$$
 \item For $k>0$,  observing that
 $\la'_N({\boldsymbol\rho}_{(1)}^{\al_1}\cdots {\boldsymbol\rho}_{(k-1)}^{\al_{(k-1)}})$ equals $\cD_{(1)}^{\beta_1}\cdots \cD_{(k-1)}^{\beta_{(k-1)}}$ plus $K$-invariant terms of lower order in $v$,
we have
\begin{equation}\label{lambda'(Phiorho)}
\begin{aligned} 
\lambda'_N(\Phi_{k,j}\circ {\boldsymbol\rho})&=\sum_{\ell=1}^{d_{k_0}}\Phi_{k,j,\ell}(\cD_{(0)})D_{k,\ell}+\sum_\al\Psi_{k,j,\al}(\cD_{(0)})\big)\la'_N({\boldsymbol\rho}_{(1)}^{\al_1}\cdots {\boldsymbol\rho}_{(k-1)}^{\al_{(k-1)}})\\
&=\sum_{\ell=1}^{d_{k_0}}\Phi_{k,j,\ell}(\cD_{(0)})D_{k,\ell}+\sum_\beta\Psi'_{k,j,\beta}(\cD_{(0)})\cD_{(1)}^{\beta_1}\cdots \cD_{(k-1)}^{\beta_{(k-1)}},
\end{aligned}
\end{equation}
where $\Psi'_{k,j,\beta}\in\cO(\bR^{d_{(0)}})$ and the last summation ranges over the multiindices $\beta$ such that the order in $v$ of $\cD_{(1)}^{\beta_1}\cdots \cD_{(k-1)}^{\beta_{(k-1)}}$ is not greater than $k$ (recall that the order is meant in the usual sense and not as a degree of homogeneity).
 \end{enumerate}

We will need the following observations concerning the interactions of the extended symmetrisations $\lambda'_N$, $\la'_{N_t}$ with $\cR^t$.
 
 For  $\psi\in \cP(\fv)\otimes \cO(\fw)$ we define $\cR^t \la'_N (\psi)$ via the same formula as in \eqref{opradon*1}:
$$
\cR^t \la'_N (\psi) = \la'_{N_t} (\psi_{|\fn_t}).
$$
It is not hard to verify that the last two identities in \eqref{composition-radon} remain true if $D$ (resp. $D_1,D_2$) is replaced by $\la'_N(\psi)$ (resp. $\la'_N(\psi_1),\la'_N(\psi_2)$)  with $\psi,\psi_1,\psi_2\in \cP(\fv)\otimes \cO(\fw)$.

Furthermore, if we apply $\cR^t \big(\la'_N (\psi)\big)$ to a spherical function $\ph^t_{\zeta,\omega,\mu}$
as in \eqref{trace} with $\zeta\in U$,
it follows from Lemma \ref{lambda-spherical} that
$$
\cR^t \big(\la'_N (\psi)\big) \ph^t_{\zeta,\omega,\mu}
= \la'_{N_t} (\psi_{|\fn_t})\ph^t_{\zeta,\omega,\mu}
= \la'_{N_t} (\chi\  \psi_{|\fn'})\ph^t_{\zeta,\omega,\mu}
=\cR^t\big( \la'_N (\chi^{\rm rad}\psi)\big) \ph^t_{\zeta,\omega,\mu}
$$
where $\chi\in C^\infty_c(\fn_t)$ is $K_t$-invariant and equal to 1 on  $\fv\times U$ and supported in $\fv\times (t+S_t)$.
Since $\chi \ \Phi \circ {\boldsymbol\rho} _{|\fn_t} = \chi \ {\boldsymbol\rho}^t $ by Proposition \ref{Qinv}
and $\la'_{N_t}({\boldsymbol\rho}^t)=\cD^t$, 
the observations above yield
$$
\cR^t \big(\la'_N (\Phi\circ {\boldsymbol\rho}_{|_{\fn_t}})\big) \ph^t_{\zeta,\omega,\mu}
=
\cD^t\ph^t_{\zeta,\omega,\mu},
\qquad \xi\in U.
$$
We can also extend \eqref{radon-eigenvalues} to obtain
that  the $\cD^t$-eigenvalues $\xi^t$
of $\ph^t=\ph^t_{\zeta,\omega,\mu}$
are equal to the eigenvalues of $\ph=\ph_{\zeta,\omega,\mu}$
for $\la'_N (\Phi\circ {\boldsymbol\rho})$.
Consequently, $\xi^t=(\xi^t_{(k)})_k\in \bR^{d_t}$ 
satisfies
\begin{equation}
\label{eq_xit2xi}
\begin{cases}
\xi^t_{(0)}=\Phi_{(0)}(\xi_{(0)})&\\
\xi^t_{k,j}=\sum_{\ell=1}^{d_{(k)}}\Phi_{k,j,\ell}(\xi_{(0)})\xi_{k,\ell}+\sum_\beta\Psi'_{k,j,\beta}(\xi_{(0)})\xi_{(1)}^{\beta_1}\cdots \xi_{(k-1)}^{\beta_{(k-1)}},&(k>0)
\end{cases} 
\end{equation}
 where  $\xi_{(k)}=(x_{k,1},\ldots,x_{k,d_{(k)}})$. The functions $\Phi_{k,j,\ell}$ and $\Psi'_{k,j,\beta}$ are slowly increasing and were given in Proposition \ref{Qinv} and in \eqref{lambda'(Phiorho)} respectively.

We have thus obtained the  expression for the local inverse of $\Lambda^t$ announced after Proposition \ref{bijection}:
\begin{lemma}
\label{lem_xit2xi}
We keep the notation above
and realise $\Lambda^t$ as a map from  $\Sigma^t_{\cD^t}$
to $\Sigma_\cD$, 
i.e.
$\Lambda^t(\xi^t)= \xi$
with $\xi^t=(\xi^t_{(k)})_{k\in \bN} \in \Sigma^t_{\cD^t}$
and $\xi=(\xi_{(k)})_{k\in \bN}\in \Sigma_\cD$;
the splittings are as in \eqref{eq_splitting_k} and \eqref{eq_splitting_kt}.
Let $A_{t,\cD}\subset \Sigma_\cD$, $A^t_{t,\cD}\subset\Sigma^t_{\cD^t}$ the sets introduced in Proposition \ref{bijection}, so that there exists $(\Lambda^t)\inv:A_{t,\cD}\longrightarrow A^t_{t,\cD}$.  
Then  the components of $\xi^t=(\Lambda^t)\inv(\xi)$ are given by \eqref{eq_xit2xi}.
\end{lemma}

\section{Extending Gelfand transforms on $\cS_0(N)^K$}
\label{sec_towards}

The  analysis developed in Section \ref{sec_quotient_pairs}
allows us to give a first result towards property (S).
More precisely, we will be concerned with the space  $\cS_0(N)$ of Schwartz functions with vanishing moments of any order in the $\fw_0$-variables. In order to define this notion, we recall the decomposition $\fw=\fw_0+\check\fw$ in \eqref{z_0}, where $\fw_0$ denotes the orthogonal complement in $\fw$ to the subspace $\check\fw$ of $K$-fixed elements. We will write $w\in\fw$ as $w_0+u$ with $w_0\in\fw_0$, $u\in\check\fw$. 
We then say that a function $F$ has {\it vanishing moment of order $\beta\in\bN^{d_{\fw_0}}$  in the $\fw_0$-variables} if
$$
\int_{\fw_0}w_0^\beta F(v,w_0,u)\,dw_0=0\ ,
$$
for every $v\in\fv$, $u\in \check\fw$. 

In this section we prove the following statement.

\begin{proposition}\label{S_0}
Let $(N,K)$ be a nilpotent Gelfand pair.
We assume that all the quotient pairs $(N_t,K_t)$ for $t\in \fw_0\setminus\{0\}$ satisfies (S). Let $\cD$ be a generating system in $\bD(N)^K$.

If $F\in\cS_0(N)^K$, then its spherical transform can be extended from $\Sigma_\cD$ to a function $f\in\cS(\bR^d)$.

Furthermore, if we choose
$\cD=\la'({\boldsymbol \rho})$
where
 ${\boldsymbol \rho} = ({\boldsymbol \rho}_{\fw_0},{\boldsymbol \rho}_{\check\fw},
{\boldsymbol \rho}_{\fv},{\boldsymbol \rho}_{\fv,\fw_0})$
is   a Hilbert basis
 satisfying (i)-(iv) in Section \ref{subsec_hilbert_base},  
 then $f$ can be chosen vanishing with all its derivatives on $\{0\}\times\bR^{d_{\check\fw}}\times\bR^{d_\fv}\times\bR^{d_{\fv,\fw_0}}$. 
\end{proposition} 

In Section \ref{subsec_Pi}, 
we define the projection of the Gelfand spectrum onto ${\boldsymbol \rho}_{\fw}(\fw)$. 
In Section \ref{partitions}, we prepare a technical tool which yields a partition of unity on ${\boldsymbol \rho}_{\fw_0}(\fw_0)$.
In Section \ref{S_0-characterization}, we give equivalent descriptions of $\cS_0(N)^K$.
Eventually in Section \ref{section-Radon}, we prove Proposition \ref{S_0}.

\subsection{The projection $\Pi$}
\label{subsec_Pi}

If $q$ is a polynomial on $\fw$, then  $\la'(q)=q(i\inv\nabla_\fw)$
by Lemma \ref{lem_sym} Part \eqref{item_lem_sym_pdtz}
and one checks readily that for any bounded spherical function 
 $\ph=\ph_{\zeta,\omega,\mu}$
given via \eqref{trace}, we have
$$
\la'(q)\ph = q(i\inv\nabla_\fw) \ph = q(\zeta)\ph.
$$
Applying this to the polynomials in ${\boldsymbol\rho}_\fw$, we obtain
$$
\xi_\fw(\ph)={\boldsymbol\rho}_\fw(\zeta)\ ,
\quad \ph=\ph_{\zeta,\omega,\mu}\in \Sigma.
$$
This gives the following.

\begin{lemma}\label{Pi}
Let ${\boldsymbol\rho}$ be  a bi-homogeneous Hilbert basis ${\boldsymbol\rho}$ of a nilpotent Gelfand pair $(N,K)$
which splits as ${\boldsymbol\rho}=({\boldsymbol\rho}_{\fw_0},{\boldsymbol\rho}_{\check\fw},{\boldsymbol\rho}_\fv,{\boldsymbol\rho}_{\fv,\fw_0})$. 
The canonical projection $\Pi$ from $\bR^d$ onto $\bR^{d_\fw}$ restricts to a surjective map
$$
\Pi_{|_{\Sigma_\cD}}:\Sigma_\cD\longrightarrow {\boldsymbol\rho}_\fw(\fw)\ ,
$$
and, for $\xi_\fw\in{\boldsymbol\rho}_\fw(\fw)$, 
$$
\Pi_{|_{\Sigma_\cD}}\inv(\xi_\fw)=\big\{\xi(\ph_{\zeta,\omega,\mu}):{\boldsymbol\rho}_\fw(\zeta)=\xi_\fw\big\}\ .
$$

In particular, the map which assigns to a spherical function $\ph_{\zeta,\omega,\mu}\in\Sigma(N,K)$ the value ${\boldsymbol\rho}_\fw(\zeta)\in \bR^{d_\fw}$ is continuous.
\end{lemma}

As ${\boldsymbol\rho}_\fw$ is a Hilbert basis on $\fw$ for the action of $K$, 
${\boldsymbol\rho}_\fw(\fw)$ is homeomorphic to the orbit space $\fw/K$.
Since 
${\boldsymbol\rho}_\fw=({\boldsymbol\rho}_{\fw_0},{\boldsymbol\rho}_{\check\fw})$
and  ${\boldsymbol\rho}_{\check\fw}$ consists of a set of coordinate functions on $\check\fw$, 
we have:
$$
{\boldsymbol\rho}_\fw(\fw)
=
{\boldsymbol\rho}_{\fw_0}(\fw_0)
\times 
\bR^{d_{\check \fw}}, 
$$
and 
${\boldsymbol\rho}_{\fw_0}(\fw_0)$ is homeomorphic to the orbit space $\fw_0/K$.

As the polynomials in ${\boldsymbol\rho}_{\fw_0}$ are homogeneous, 
$0\in  {\boldsymbol\rho}_{\fw_0}(\fw_0)$.
The pre-image 
\begin{equation}
\label{eq_def_almost_most_singular}
\Pi_{|_{\Sigma_\cD}}^{-1}  
\left(   \{0\}\times \bR^{d_{\check \fw}} \right),
\end{equation}
in $\Sigma_\cD$ of $\{0\}\times \bR^{d_{\check\fw}}\subset 
 {\boldsymbol\rho}_{\fw_0}(\fw_0)\times \bR^{d_{\check\fw}}$ under $\Pi$ 
will play a special r\^ole in the proof of Proposition~\ref{S_0}.

\subsection{Partitions of unity}\label{partitions}

The constructions in Section \ref{subsec_rad_quotient}
 present a natural homogeneity with respect to the dilations on $\fw_0$, as well as translation-invariance  with respect to $\check\fw$. Precisely if, for $t_0\in \fw_0\setminus\{0\}$, the conclusions of Proposition \ref{Qinv} are satisfied on a $\fw_{t_0}$-neighbourhood $U\subset t_0+S_{t_0}$ of $t_0$, then they are also satisfied
\begin{enumerate}
\item[(i)]  on the neighbourhood $\del U$ of $\del t_0$, for $\del>0$;
\item [(ii)] on the neighbourhood $ U+u$ of $t_0+u$, for $u\in\check\fw$.
\end{enumerate}

Also notice that, since $K$ acts trivially on $\check\fw$, we may assume that the $\check\fw$-variables do not appear in components of ${\boldsymbol\rho}$ other than ${\boldsymbol\rho}_{\check\fw}$. Since $\check\fw\subset\fw_t$ for every $t\in\fw$, we may also assume  the same  on the components of ${\boldsymbol\rho}^t$ for every $t\in\fw\setminus\check\fw$.

We denote by $\cQ_t :\bR^{d_t}\to \bR^d$
the map given via \eqref{eq_xi2xit}, 
that is, 
$$
\cQ_{t,0,j}(\xi^t)=
Q_{0,j}(\xi_{(0)}^t) 
\quad\mbox{and}\quad
\cQ_{t,k,j}(\xi^t)  = 
\sum_{\ell=1}^{d_{(k)}^t}
Q_{k,j,\ell}(\xi_{(0)}^t) \xi_{k,\ell} ^t
\ +\ 
R'_{k,j} (\xi_{(0)} ^t,\ldots, \xi_{(k)}^t).
$$
We denote by $\Phi_t :\bR^{d}\to \bR^{d_t}$
the map in \eqref{eq_xit2xi}, 
i.e., 
\begin{equation}
\label{eq_def_Phit}
\Phi_{t,0}(\xi)=\Phi_{(0)}(\xi_{(0)})
\quad\mbox{and}\quad
\Phi_{t,k,j}(\xi)=\sum_{\ell=1}^{d_{k_0}}\Phi_{k,j,\ell}(\xi_{(0)})\xi_{k,\ell}+\Psi_{k,j}'(\xi_{(0)},\ldots, \xi_{(k-1)}).
\end{equation}
We denote by  $D(\del)$ be the dilations \eqref{dilations} on $\bR^d$ with exponents $\nu_j$ equal to the degrees of homogeneity of the elements of ${\boldsymbol\rho}$.
Similarly, $D^t(\del)$ denotes the dilations on $\bR^{d_t}$ with exponents $\nu^t_j$, equal to the degrees of homogeneity of the elements of ${\boldsymbol\rho}^t$.

All this has the following implications on the maps 
constructed in Section \ref{subsec_rad_quotient}.

\begin{enumerate}
\item[(i)] The maps $\cQ_t$, $\Phi_t$ contain the identity function in the $\xi_{\check\fw}$-component, and all the other components do not involve the $\xi_{\check\fw}$-variables.
\item[(ii)] In Section \ref{subsec_rad_quotient}, the slice $S_t$ at $t=t_0+u$, $t_0\ne0$, can be chosen by first taking a slice $S_{0,t_0}$ at $t_0$ in $\fw_0$ and then set $S_t=S_{0,t_0}+\check\fw$. In the same way the relatively compact neighbourhood $U_t\subset t+S_t$ can be taken of the form $U_{0,t_0}+\check\fw$. Notice that  $0\not\in U_{0,t_0}$.
\item[(iii)]  
Then, for $\del>0$, $\cQ_{\del t}=\cQ_t$ and
$D(\del)\circ \cQ_t=\cQ_t\circ D^t(\del)$.
\item[(iv)] Once $\Phi_t$ has been chosen for $t\in\fw$ with $|t|=1$, $\Phi_{\del t}$ can be chosen, for $\del>0$,  as
\begin{equation}\label{delta-scaling}
\Phi_{\del t}=D^t(\del)\circ \Phi_t\circ D(\del\inv) .
\end{equation}
\end{enumerate}

Let  $T$ be a finite set of points on the unit sphere in $\fw_0$ such that $\{KU_{0,t}\}_{t\in T}$ covers the unit sphere. Then there is $r>1$  such that the annulus $\{w_0\in\fw_0:1\le|w_0|\le r\}$ is contained in $ \bigcup_{t\in T}KU_{0,t}$. Therefore $\{r^jKU_{0,t}\}_{t\in T,\,j\in\bZ}$ is a locally finite covering of $\fw_0\setminus\{0\}$.

For each $t\in T$ we choose $\chi_t\ge0$ in $C^\infty_c(\fw_t\cap\fw_0)$ supported on $U_{0,t}$ so that $\sum_t\chi_t>0$ on $\{w\in\fw_0:1\le|w|\le r\}$. By Corollary \ref{Slice-cor}, we can define $\chi^\#_{t,j}\in C^\infty_c(\fw_0)$ supported in $r^j KU_{0,t}$ and such that
$$
\chi^\#_{t,j}(w)=\chi_t^{\rm rad}(r^{-j}w)
$$ 
on 

Up to dividing each $\chi_{t,j}^\#$ by $\sum_{t\in T,\,j\in\bZ}\chi_{t,j}^\#$, we may assume that the
$\chi_{t,j}^\#$ form a partition of unity on $\fw_0\setminus\{0\}$  subordinated to the covering $\{r^jKU_{0,t}\}_{t\in T,j\in\bZ}$.

\begin{lemma}\label{xi-partition}
There exists a family  $\{\eta_{t,j}\}_{t\in T,\,j\in\bZ}$ of nonnegative functions on $\bR^{d_{\fw_0}}$ and a $D(\del)$-invariant neighbourhood $\Omega$ of ${\boldsymbol\rho}_{\fw_0}(\fw_0)\setminus\{0\}$ in $\bR^{d_{\fw_0}}$ such that
\begin{equation}\label{sum-eta}
\sum_{t\in T,\,j\in\bZ}\eta_{t,j}(\xi)=1 \ ,\qquad (\xi\in\Omega)\ , 
\end{equation}
and, for every $t,j$,
\begin{enumerate}
\item[\rm(i)]  $\eta_{t,j}\in C^\infty_c(\bR^{d_{\fw_0}}\setminus\{0\})$;
\item[\rm(ii)] $\eta_{t,j}(\xi)=\eta_{t,0}\big(D(r^{-j})\xi\big)$;
\item[\rm(iii)] $(\supp\eta_{t,j})\cap {\boldsymbol\rho}_{\fw_0}(\fw_0)\subset {\boldsymbol\rho}_{\fw_0}(r^jU_{0,t})$;
\item[\rm(iv)]  $\eta_{t,j}\big({\boldsymbol\rho}_{\fw_0}(w)\big)=\chi^\#_{t,j}(w)$ for all $w\in\fw_0$.
\end{enumerate}
\end{lemma}

\begin{proof}
By \cite{Schw}, there exist smooth functions $u_t$, $t\in T$, on $\bR^{d_{\fw_0}}$ such that 
$$
u_t\big(\boldsymbol\rho_{\fw_0}(w)\big)=\chi_t^{\rm rad}(w)=\chi^\#_{t,0}(w)\ .
$$

Setting $u_{t,j}=u_t\circ D(r^{-j})$, we have
$\chi_{t,j}^\#(w)=u_{t,j}\big({\boldsymbol\rho}_{\fw_0}(w)\big)$, 
for every $t$ and $j$.

Since ${\boldsymbol\rho}_{\fw_0}(\supp\chi^\#_{t,0})$ does not contain the origin, we may assume that each $u_t$, $t\in T$, is supported on a fixed compact set $E$ of $\bR^{d_{\fw_0}}$ not containing the origin. Moreover, since ${\boldsymbol\rho}_{\fw_0}$ is a proper map, and ${\boldsymbol\rho}_{\fw_0}(\fw_0\setminus KU_{0,t})$ is closed in $\bR^{d_{\fw_0}}$ and disjoint from ${\boldsymbol\rho}_{\fw_0}(\supp\chi^\#_{t,0})$, we may also assume  that $(\supp u_t)\cap{\boldsymbol\rho}_{\fw_0}(\fw_0)\subset {\boldsymbol\rho}_{\fw_0}(KU_{0,t})={\boldsymbol\rho}_{\fw_0}(U_{0,t})$.

Clearly,
$\sum_{t,j}u_{t,j}=1$ on ${\boldsymbol\rho}_{\fw_0}(\fw_0)\setminus\{0\}$. Therefore, if we set
$$
\eta_{t,j}=\frac{u_{t,j}}{\sum_{t',j'}u_{t',j'}}\ ,
$$
$\eta_{t,j}=u_{t,j}$ on ${\boldsymbol\rho}_{\fw_0}(\fw_0)$, and the sum of the $\eta_{t,j}$ remains equal to 1 where some $\eta_{t,j}$ is positive.  Then \eqref{sum-eta} and properties (i)-(iv) follow easily.
\end{proof}

\subsection{Characterisations of functions in $\cS_0(N)^K$}\label{S_0-characterization}
If  $g$ is an  integrable function  on $\fw_0$ and $F$ and integrable function on $N$, we set
\begin{equation}\label{*w_0}
F*_{\fw_0}g(v,w_0,u)=\int_{\fw_0} F(v,w_0-w'_0,u)g(w'_0)\,dw'_0\ .
\end{equation}

This  can be regarded as the convolution  on $N$ of $F$ and the finite measure $\del_0\otimes g$, where $\del_0$ is the Dirac delta at the origin in $\fv\oplus\check\fw$.
We use the symbol $\widehat{\phantom a}$ to denote Fourier transform in the $\fw_0$-variables.  For a  function on $N$ we then set
\begin{equation}\label{fourier}
\widehat F(v,\zeta,u)=\int_{\fw_0} F(v,w_0,u)e^{-i\lan w_0,\zeta\ran}\,dw_0\ ,
\end{equation}
for $v\in\fv$, $\zeta\in\fw_0$, $u\in\check\fw$. For $F$ and $g$  as in \eqref{*w_0}, 
$\widehat{F*_{\fw_0}g}(v,\zeta,u)=\widehat F(v,\zeta,u)\widehat g(\zeta)$.

Finally, we denote by $\psi_{t,j}$ the inverse Fourier transform of $\chi^\#_{t,j}$.

\begin{lemma}\label{S_0-equivalence}
The following are equivalent for a function $F\in\cS(N)$:
\begin{enumerate}
\item[\rm(i)] $F\in\cS_0(N)$;
\item[\rm(ii)] $\widehat F(v,\zeta,u)$ vanishes with all its derivatives for $\zeta=0$;
\item [\rm(iii)] for every $k\in\bN$, $F(v,w)=\sum_{|\al|=k}\de_w^\al G_\al(v,w)$, with $G_\al\in\cS(N)$ for every $\al$;
\item[\rm(iv)] the series $\sum_{t\in T,\,j\in\bZ}F*_{\fw_0}\psi_{t,j}$ converges to $F$ in every Schwartz norm;
\item[\rm(v)] for every Schwartz norm $\|\ \|_{\cS(N),M}$ and every $q\in\bN$, $\|F*_{\fw_0}\psi_{t,j}\|_{\cS(N),M}=o(r^{-q|j|})$ as $j\to\pm\infty$.
\end{enumerate}
\end{lemma}

\begin{proof}
The equivalence of (i) and (ii) is a direct consequence of the definition of $\cS_0(N)$. The equivalence of (ii) and (iii) follows from Hadamard's lemma, cf. \cite{FRY2}. Finally, the equivalence among (ii), (iv) and (v) can be easily seen on the $\fw_0$-Fourier transform side.
\end{proof}

\subsection{Proof of Proposition \ref{S_0}}
\label{section-Radon}

Let $(N,K)$ be a nilpotent Gelfand pair.
We assume that all the quotient pairs $(N_t,K_t)$ for $t\in \fw_0\setminus\{0\}$ satisfies (S).
We continue with the notation of 
Sections \ref{partitions} and \ref{subsec_rad_quotient}.

Let  $F\in\cS_0(N)^K$.
Decompose $F$ according to Lemma \ref{S_0-equivalence} (iv). Then
$$
\cG F=\sum_{t\in T,\,j\in\bZ}\cG(F*_{\fw_0}\psi_{t,j})\ .
$$

Denoting by $\mu_{t,j}$ the measure $\del_0\otimes\psi_{t,j}$, where $\del_0$ is the Dirac delta at the origin in $\fv\oplus\check\fw$, $\cG(F*_{\fw_0}\psi_{t,j})$ is the product of $\cG F$ and $\cG\mu_{t,j}$. By \eqref{trace}, if $\zeta=\zeta_0+\check\zeta\in\fw_0\oplus\check\fw$, then 
$$
\begin{aligned}
\cG\mu_{t,j}(\ph_{\zeta,\omega,\mu})&=\int_{\fw_0}\psi_{t,j}(w_0)\ph_{\zeta\omega,\mu}(0,-w_0,0)\,dw_0\\
&=\int_{\fw_0}\int_K\psi_{t,j}(w_0)e^{-i\lan\zeta_0,kw_0\ran}\,dk\,dw_0\\
&=\chi_{t,j}^\#(\zeta_0)\ .
\end{aligned}
$$

Then, for $\xi=(\xi_{\fw_0},\xi_{\check\fw},\xi_\fv,\xi_{\fv,\fw_0})\in\Sigma_\cD$,
\begin{equation}\label{G(f*psi)}
\cG(F*_{\fw_0}\psi_{t,j})(\xi)=\cG F(\xi)\eta_{t,j}(\xi_{\fw_0})\ .
\end{equation}

Set $F_{t,j}=F*_{\fw_0}\psi_{t,j}$, and consider $\cR^tF_{t,j}$. Since we are assuming that property (S) holds for $(N_t,K_t)$, we have $\cG_t(\cR^tF_{t,j})\in \cS(\Sigma^t_{\cD^t})$. Moreover, for every $M,q\in\bN$, 
$$
\|\cG_t(\cR^tF_{t,j})\|_{M,\cS(\Sigma^t_{\cD^t})}=o(r^{-q|j|})\ ,
$$
by Part (v) of Lemma \ref{S_0-equivalence}
and the continuity of $\cG_t$ and $\cR^t$.

For fixed $M$, there exist functions $h^{(M)}_{t,j}\in\cS(\bR^d)$ such that 
$$
{h^{(M)}_{t,j}}_{\Sigma^t_{\cD^t}}=\cG_t(\cR^tF_{t,j})\ ,\qquad \|h^{(M)}_{t,j}\|_{M,\cS(\bR^d)}\le 2\|\cG_t(\cR^tF_{t,j})\|_{M,\cS(\Sigma^t_{\cD^t})}\ .
$$
Recall that the maps $\Phi_t$ and $\Phi_{r^jt}$
are defined via \eqref{eq_def_Phit} and \eqref{delta-scaling}.
Since $\cG F_{t,j}$ is supported in $\Sigma_\cD\cap \Pi\inv\big({\boldsymbol\rho}_{\fw_0}(r^jU_{0,t})\big)$,  \eqref{G-G_t} and Lemma \ref{lem_xi2xit} imply that the composition $g^{(M)}_{t,j}=h^{(M)}_{t,j}\circ\Phi_{r^jt}$ coincides with $\cG F_{t,j}$ on $\Sigma_{\cD}$.

Therefore, for every choice of the integers $M_j$, we can say that the series
$$
\sum_{t,t'\in T,\,j,j'\in\bZ}\eta_{t',j'}(\xi_{\fw_0})g^{(M_j)}_{t,j}(\xi)
$$
converges pointwise to $\cG F$ on $\Sigma_\cD$. In fact, many of the terms will vanish identically on $\Sigma_\cD$, and this surely occurs when ${\boldsymbol\rho}_{\fw_0}(r^{j'}U_{0,t'})\cap {\boldsymbol\rho}_{\fw_0}(r^jU_{0,t})=\emptyset$.  Therefore,  if $E_{t,j}=\{(t',j'):{\boldsymbol\rho}_{\fw_0}(r^{j'}U_{0,t'})\cap {\boldsymbol\rho}_{\fw_0}(r^jU_{0,t})\ne\emptyset\}$,
\begin{equation}\label{sum}
\sum_{t\in T,\,j\in\bZ}\sum_{(t',j')\in E_{t,j}}\eta_{t',j'}(\xi_{\fw_0})g^{(M_j)}_{t,j}(\xi)=\cG F(\xi)
\end{equation}
on $\Sigma_\cD$.
Notice that, by construction,  there is $A>0$ such that $|j-j'|\le A$ for $(t',j')\in E_{t,j}$. In particular, the sets $E_{t,j}$ are finite and their cardinalities have a uniform upper bound. 

We claim that, choosing the $M_j$ appropriately, we can make the series \eqref{sum} converge to the required Schwartz extension of $\cG F$.

 In order to estimate the Schwartz norms of $\eta_{t',j'}g^{(M)}_{t,j}$, for $(t',j')\in E_{t,j}$, observe that, by Proposition \ref{Qinv}, the $\fw_0$-variables are bounded to a compact set, and the other variables appear in $\Phi_t$ as polynomial factors. Taking this into account, together with the scaling properties of $\Phi_{r^jt}$ and $\eta_{t',j'}$, cf.  \eqref{delta-scaling} and Lemma \ref{xi-partition} (ii), it is not hard to see that there is $p_M$ depending only on $M$ such that, for $(t',j')\in E_{t,j}$,
$$
\|\eta_{t',j'}g^{(p_M)}_{t,j}\|_{M,\cS(\bR^d)}\le Cr^{|j|p_M}\|h^{(p_M)}_{t,j}\|_{p_M,\cS(\bR^d)}\ .
$$

Hence,
$$
\|\eta_{t',j'}g^{(p_M)}_{t,j}\|_{M,\cS(\bR^d)}=o(r^{-q|j|})\ ,
$$
for every $q$ and $(t',j')\in E_{t,j}$.

By induction, we can then select a strictly increasing sequence $\{j_q\}_{q\in\bN}$ of integers, such that $j_0=0$ and
$$
\|\eta_{t',j'}g^{(p_q)}_{t,j}\|_{q,\cS(\bR^d)}\le r^{-q|j|}\ ,
$$
for $j\ge j_q$ and $(t',j')\in E_{t,j}$. For $j_q\le j< j_{q+1}$, we select $M_j=p_q$ and consider the following special case of~\eqref{sum}:
$$
g(\xi)=\sum_{q=0}^\infty\,\sum_{\substack{j_q\le j<j_{q+1}\\ t\in T}}\,\sum_{(t',j')\in E_{t,j}}\eta_{t',j'}(\xi_{\fw_0})g^{(p_q)}_{t,j}(\xi)\ .
$$

Then the series converges in every Schwartz norm. 
This defines $g\in \cS(\bR^d)$.
We check easily that $g$ coincides with $\cG F$ on $\Sigma_\cD$.
This shows the first part of Proposition \ref{S_0}, 
as the corresponding property does not depend on a choice of $\cD$,
see Lemma \ref{lem_poly_map_SigmacDs}.

Since each term vanishes identically on a neighbourhood of $\{0\}\times\bR^{d_{\check\fw}}\times\bR^{d_\fv}\times\bR^{d_{\fv,\fw_0}}$, the sum must have all derivatives vanishing on this set.
This concludes the proof of Proposition~\ref{S_0}.

\section{The pair $(\check N,K)$ and hypothesis (H)}
\label{sec_checkN+H}

Section \ref{sec_towards} shows the Schwartz extension property 
we want but only
for functions in $\cS_0(N)^K$
and 
provided that all the proper quotient pairs 
satisfy (S).
We now show that under certain conditions, 
we can subtract from a general function $F\in\cS(N)^K$ a function $G$ with Schwartz spherical transform  so 
that $F-G\in \cS_0(N)^K$.

\begin{proposition}\label{F-G}
Let $(N,K)$ be a nilpotent Gelfand pair satisfying hypothesis (H).
Given $F\in\cS(N)^K$, there is $G\in\cG\inv\big(\cS(\bR^d)\big)$ such that $F-G\in\cS_0(N)^K$.
\end{proposition}

The condition named hypothesis (H) will be formulated in Section \ref{subsec_hypH}.
In Section \ref{subsec_dominant}, we explain our concept of dominant variable on the Gelfand spectrum; this concept  will be  essential in the proofs of Proposition \ref{hadamard} and Proposition \ref{F-G} 
in Sections \ref{subsec_pf_prop_hadamard} and \ref{subsec_pf_propF-G}
respectively.

\subsection{Hypothesis (H)}
\label{subsec_hypH}

Let $(N,K)$ be a nilpotent Gelfand pair.
We consider $\check \fn = \fv\oplus\check \fw$
and  the nilpotent Gelfand pair  $(\check N,K)$ obtained by central reduction of $\fw_0$. By Proposition \ref{trivial}, $(\check N,K)$ satisfies property (S).

We fix a Hilbert basis ${\boldsymbol\rho}=({\boldsymbol\rho}_{\fw_0},{\boldsymbol\rho}_{\check\fw},{\boldsymbol\rho}_\fv,{\boldsymbol\rho}_{\fv,\fw_0})$ as in Section \ref{subsec_hilbert_base}.
For a multi-index $\al=(\al',\al'')\in\bN^{d_{\fw_0}}\times\bN^{d_{\fv,\fw_0}}$, we denote by  $[\al]$ the degree of   ${\boldsymbol\rho}_{\fw_0}^{\al'}{\boldsymbol\rho}_{\fv,\fw_0}^{\al''}$ in the $\fw_0$-variables.

For $\rho_j\in{\boldsymbol\rho}_{\fw_0}\cup{\boldsymbol\rho}_{\fv,\fw_0}$, let $\tilde D_j\in\bD(\check N)\otimes\cP(\fw_0)$ denote $(\la'_{\check N}\otimes I)(\rho_j)$, where $\la'_{\check N}$ is the symmetrisation operator \eqref{modsym} for $\check N$, and $\rho_j$ is regarded as an element of $\cP(\check\fn)\otimes\cP(\fw_0)$. 
If $\al=(\al',\al'')\in\bN^{d_{\fw_0}}\times\bN^{d_{\fv,\fw_0}}$, then
$$
\tilde D^\al={\boldsymbol\rho}_{\fw_0}^{\al'}\tilde D^{\al''}
$$
has degree $[\al]$ in the $\fw_0$-variables.

\medskip

\begin{definition}\label{def-H} 
We say that $(N,K)$ satisfies {\rm hypothesis (H)} if, for any  $K$-invariant function $G$ on $\check N\times\fw_0$ of the form
\begin{equation}\label{G}
 G(v,w_0,u)=\sum_{|\gamma|=k}w_0^\gamma G_\gamma(v,u)\ ,
\end{equation}
with $G_\gamma\in\cS(\check N)$, 
 there exist functions $H_{\al}\in\cS(\check N)^K$, for $[\al]=k$, such that
\begin{equation}\label{Hal}
 G=\sum_{[\al]= k}\frac1{\al!}\tilde D^{\al} H_{\al}\ .
\end{equation}
\end{definition}

The following statement has been essentially proved in  \cite{FRY2}. We give the explicit proof for completeness.

\begin{proposition}\label{H-checkNabelian}
If $\check\fw$ is trivial, i.e. if $\check N$ is abelian, hypothesis (H) is satisfied.
\end{proposition}

\begin{proof}
Under our hypotheses we are given $G(v,w_0)=\sum_{|\gamma|=k}w_0^\gamma G_\gamma(v)\in \big(\cS(\fv)\otimes\cP^k(\fw_0)\big)^K$. The same holds for the partial Fourier transform in $\fv$,
$$
\cF_\fv G(v,w_0)=\sum_{|\gamma|=k}w_0^\gamma \widehat G_\gamma(v).
$$

By Lemma \ref{tensor-invariants}, adapted to Schwarts functions, cf. \cite[Prop. 2.1]{FRY2},
$$
\cF_\fv G(v,w_0)=\sum_\al u_\al({\boldsymbol\rho}_{\fw_0}){\boldsymbol\rho}_{\fv}^{\al_1}{\boldsymbol\rho}_{\fv,\fw_0}^{\al_2}.
$$

Undoing the Fourier transform, we obtain \eqref{Hal}.
\end{proof}

Hypothesis (H) was proved in \cite{FRY2} for the nilpotent Gelfand pairs in Vinberg's classification~\cite{V2}, including those listed in Table \eqref{table_block1+2} and discussed in Section \ref{sec-application}.

At the heart of the proof of Proposition \ref{F-G} is the following property.
\begin{proposition}\label{hadamard}
Let $(N,K)$ be a nilpotent Gelfand pair satisfying hypothesis (H).
We keep the notation of hypothesis (H).
We also consider the family of operators $\cD=(\cD_{\check \fw},\cD_{\fw_0},\cD_{\fv},\cD_{\fv,\fw_0})$ obtained by symmetrisation
of ${\boldsymbol\rho}$ via $\la'_N$.
Let $F\in\cS(N)^K$, and assume that 
$$
F(v,w_0,u)=\sum_{|\gamma|=k}\de_{w_0}^\gamma R_\gamma(v,w_0,u)\ ,
$$
with $R_\gamma\in\cS(N)$ for every $\gamma$. Then, for every $\al$ with $[\al]=k$, there exists a function $F_\al\in\cS(N)^K$ such that
$$
\cG F_\al(\xi)=h_\al(\xi_{\check\fw},\xi_\fv)\ ,\qquad (\xi\in\Sigma_\cD)\ ,
$$
with $h_\al\in\cS(\bR^{d_{\check \fw}+d_{\fv}})$, and
$$
F(v,w_0,u)=\sum_{[\al]=k}\frac1{\al!}D^\al F_\al(v,w_0,u)+\sum_{|\gamma'|=k+1}\de_w^\beta R'_{\gamma'}(v,w_0,u)\ ,
$$
with $R'_{\gamma'}\in\cS(N)$ for every $\gamma'$.
\end{proposition}

The proof of Proposition \ref{hadamard} will be given in Section
\ref{subsec_pf_prop_hadamard}. It relies on the notion of dominant variable in the Gelfand spectrum explained in the next section.

Before going into it, we need the following notation and observations.
Consider the transform  $\check \cR$ defined as in \eqref{radon*} and \eqref{opradon*} with $\fs=\fw_0$ and mapping
$\cS(N)\to \cS(\check N)$ and $\bD(N)$ to $\bD(\check N)$.
As the restrictions to $\check \fn$ of 
${\boldsymbol \rho}_{\fw_0}$ and 
${\boldsymbol \rho}_{\fv,\fw_0}$ are zero, 
by \eqref{opradon*1}, we have:
$$
\check \cR \cD_{\fw_0}=0
\quad\mbox{and}\quad
\check \cR \cD_{\fv,\fw_0}=0.
$$
Recall that ${\boldsymbol \rho}_{\check \fw}$ is a system of coordinates of $\check \fw$ since the action of $K$ on $\check \fw$ is trivial, and 
that ${\boldsymbol \rho}_{\fv}$ is a Hilbert basis for the action of $K$ on $\fv$.
Hence ${\boldsymbol \rho}_{\fv}$, ${\boldsymbol \rho}_{\check \fw}$
form a Hilbert basis $\check {\boldsymbol \rho}$ for $(\check N, K)$.
The operators obtained by symmetrisation are by \eqref{opradon*1}
$$
\check \cD_{\fv} = \check\cR \cD_{\fv}
\quad\mbox{and}\quad
\check \cD_{\check \fw} = \check\cR \cD_{\check \fw} = i^{\inv} \nabla_{\check \fw},
$$
and form a generating family $\check \cD=(\check \cD_{\fv},\check \cD_{\check \fw} )$ of $\bD(N)^K$.

Denote by $\Sigma$, $\check\Sigma$ the  Gelfand spectra of $(N,K)$ and $(\check N, K)$
and by $\Sigma_\cD$, $\check\Sigma_{\check\cD}$
their realisations via $\cD$ and $\check \cD$ respectively.
The proof of Proposition \ref{central}
given in Section \ref{subsec_pf_prop_central}
implies that we may consider 
$\Sigma_{\check\cD}$ as the following subset of $\Sigma$:
\begin{equation}\label{check-sigma}
\Sigma_{\check\cD}=\big\{(0,\xi_{\check\fw},\xi_\fv,0):(\xi_{\check\fw},\xi_\fv)\in\check\Sigma_{\check\cD}\big\}\ ;
\end{equation}
Moreover, denoting by  $\check\cG$ the Gelfand transform of $(\check N,K)$, 
we have the identity
\begin{equation}\label{gelfand-radon}
\check\cG(\check\cR F)(\xi_{\check\fw},\xi_\fv)=\cG F(0,\xi_{\check\fw},\xi_\fv,0)\ .
\end{equation}

\subsection{Dominant variables of the Gelfand spectrum}
\label{subsec_dominant}

Let $(N,K)$ be a nilpotent Gelfand pair.
We fix a Hilbert basis ${\boldsymbol\rho}=({\boldsymbol\rho}_{\fw_0},{\boldsymbol\rho}_{\check\fw},{\boldsymbol\rho}_\fv,{\boldsymbol\rho}_{\fv,\fw_0})$ as in Section \ref{subsec_hilbert_base} and
 the corresponding family of operators $\cD=(\cD_{\check \fw},\cD_{\fw_0},\cD_{\fv},\cD_{\fv,\fw_0})$ obtained by symmetrisation.

We have already observed that the choice of bi-homogeneous Hilbert basis 
yields natural dilations $D(\del)$ given in \eqref{dilations} on the realisation $\Sigma_\cD$ of the spectrum. This leads us to 
introduce the homogeneous norm on $\bR^d$ , compatible with the dilations $D(\del)$,
\begin{equation}\label{norm}
\|\xi\|=\sum_{j=1}^d|\xi_j|^{\frac1{\nu_j}}\ .
\end{equation} 

We may write any element $\xi$ of 
$\bR^d=\bR^{d_{\fw_0}}\times\bR^{d_{\check\fw}}\times\bR^{d_{\fv}}\times\bR^{d_{\fv,\fw_0}}$ 
as $\xi=\big(\xi_{\fw_0},\xi_{\check\fw},\xi_\fv,\xi_{\fv,\fw_0}\big)$
and allow ourselves to write 
$$
\|\xi_{\fw_0}\|=\|\big(\xi_{\fw_0},0,0,0\big)\|,
$$
and similarly for $\xi_{\check\fw}$, $\xi_\fv$, $\xi_{\fv,\fw_0}$.

In the sense of the next lemma, 
the variable $\xi_\fv$ dominates the other on the spectrum $\Sigma_\cD$:

\begin{lemma}\label{dominant}
Let ${\boldsymbol\rho}$ be  a bi-homogeneous Hilbert basis ${\boldsymbol\rho}$ of a nilpotent Gelfand pair $(N,K)$
which splits as ${\boldsymbol\rho}=({\boldsymbol\rho}_{\fw_0},{\boldsymbol\rho}_{\check\fw},{\boldsymbol\rho}_\fv,{\boldsymbol\rho}_{\fv,\fw_0})$. 
Let also $\cD=\lambda'({\boldsymbol\rho})$ be the associated $d$-tuple of differential operators.
Let $\xi=(\xi_{\fw_0},\xi_{\check\fw},\xi_\fv,\xi_{\fv,\fw_0})$ be a point in 
the spectrum $\Sigma_\cD$.
We have the following inequalities:
\begin{enumerate}
\item[\rm(i)] if $\rho_j\in\cP^{\nu'_j}(\fv)\otimes\cP^{\nu''_j}(\fw_0)$, with $\nu'_j,\nu''_j>0$, then $|\xi_j|\le C\|\xi_\fv\|^{\frac{\nu'_j}2}\|\xi_{\fw_0}\|^{\nu''_j}$;
\item[\rm(ii)] $\|\xi\|\le C\|\xi_\fv\|$.
\end{enumerate}

In particular, if $\xi_{\fw_0}=0$, then also $\xi_{\fv,\fw_0}=0$.
\end{lemma}

\begin{proof}
Let $p(v)=|v|^2$, where $|\ |$ denotes the norm induced by a $K$-invariant inner product on $\fv$. 
Denote by $X_v$ the left-invariant vector field equal to $\de_v$ at the identity of $N$. Then $\la'(p)$ is the sublaplacian $L=-\sum X_{e_j}^2$, where $\{e_j\}$ is an orthonormal basis of $\fv$. Then $L$ is hypoelliptic and, for every $v_1,\dots,v_m\in\fv$ and $F\in\cS(N)$,
$$
\|X_{v_1}\cdots X_{v_m}F\|_2\le C_{v_1,\dots,v_m}\| L^{\frac m2}F\|_2\ ,
$$
cf. \cite{FS}.
For $w_1,\dots,w_m\in\fw_0$ and $F\in\cS(N)$, we also have, by classical Fourier analysis,
$$
\|\de_{w_1}\cdots \de_{w_m}F\|_2\le C_{w_1,\dots,w_m}\| \Delta_{\fw_0}^{\frac m2}F\|_2\ .
$$

Therefore,
$$
\begin{aligned}
\|\de_{w_1}\cdots \de_{w_{m'}}X_{v_1}\cdots X_{v_m}F\|_2^2&=\big|\lan\de_{w_1}^2\cdots \de_{w_{m'}}^2F,X_{v_1}\cdots X_{v_m}X_{v_m}\cdots X_{v_1}F\ran\big|\\
&\le C_{v_1,\dots,v_m,w_1,\dots,w_{m'}}\|L^mF\|_2\|\Delta_{\fw_0}^{m'}F\|_2\ .
\end{aligned}
$$

If $\rho_j\in\cP^{\nu'_j}(\fv)\otimes\cP^{\nu''_j}(\fw_0)$, then $D_j=\la'(\rho_j)$ is a linear combination of terms of this kind, with $m=\nu'_j$ and $m'=\nu''_j$. Therefore, for $F\in\cS(N)$,
\begin{equation}\label{control}
\|D_jF\|_2\le C\|L^{\nu'_j}F\|_2^\half\|\Delta_{\fw_0}^{\nu''_j}F\|_2^\half\ .
\end{equation}

 We may assume that ${\boldsymbol\rho}_\fv$ contains the squares $p_k=|v_k|^2$ of the norm restricted to mutually orthogonal irreducible components of $\fv$. Then $p$ is the sum of such $p_k$. The same can be said about $q(w_0)=|w_0|^2$ on~$\fw_0$. 

Denote by $\xi_k$  the component of $\xi_\fv$ corresponding to $\la'(p_k)$, $\xi_\ell$  the component of $\xi_{\fw_0}$ corresponding to $\la'(q_\ell)$,  and $\xi_j$ the component of $\xi_{\fv,\fw_0}$ corresponding to $D_j$. Then, since $\la'(p_k)$ and $\la'(q_\ell)$ are positive operators, we have $\xi_k,\xi_\ell\ge0$ in $\Sigma_\cD$. Hence \eqref{control} implies that, on~$\Sigma_\cD$,
$$
|\xi_j|\le C\Big(\sum_k\xi_k\Big)^\frac{\nu'_j}2\Big(\sum_\ell\xi_\ell\Big)^\frac{\nu''_j}2\le C\|\xi_\fv\|^\frac{\nu'_j}2\|\xi_{\fw_0}\|^{\nu''_j}\ .
$$

This proves (i). To prove (ii) it suffices to prove that $\|\xi_\fw\|\le C\|\xi_\fv\|$. For this, it suffices to observe that every derivative $\de_w^\al$ in the $\fw$-variables can be expressed as a combination of products $X_{v_1}\cdots X_{v_m}$ with $m=2|\al|$. Then, for every $F\in\cS(N)$,
$$
\|\de_w^\al F\|_2\le C\|L^{|\al|}F\|_2\ .\qedhere
$$
\end{proof}

\subsection{Proof of Proposition \ref{hadamard}}
\label{subsec_pf_prop_hadamard}

Let $(N,K)$ be a nilpotent Gelfand pair.
We keep the notation of Sections \ref{subsec_hypH} and
 \ref{subsec_dominant}.
We  assume  that $(N,K)$  satisfies hypothesis (H) and recall
 that the nilpotent Gelfand pair $(\check N,K)$ satisfies (S).

Consider the Fourier transform \eqref{fourier} of $F$ in the $\fw_0$-variables. Then
$$
\widehat F(v,\zeta,u)=i^k\sum_{|\gamma|=k}\zeta^\gamma\widehat R_\gamma(v,\zeta,u)\ .
$$

Setting $G_\gamma(v,u)=i^k\widehat R_\gamma(v,0,u)\in\cS(\check N)$, we have
$$
\widehat F(v,\zeta,u)=\sum_{|\gamma|=k}\zeta^\gamma G_\gamma(v,u)+\sum_{|\gamma'|=k+1}\zeta^{\gamma'} S_{\gamma'}(v,\zeta,u)\ ,
$$

Since $G=\sum_\gamma \zeta^\gamma G_\gamma$ is $K$-invariant, 
as hypothesis (H) is satisfied, see Section \ref{subsec_hypH},
 there exist $H_\al\in\cS(\check N)^K$ such that
$$
\widehat F(v,\zeta,u)=\sum_{[\al]= k}\frac1{\al!}(\tilde D^{\al} H_{\al})(v,\zeta,u)+\sum_{|\gamma'|=k+1}\zeta^{\gamma'} S_{\gamma'}(v,\zeta,u)\ ,
$$

By property (S)  for the pair $(\check N,K)$, for every $\al$ there is $h_\al\in\cS(\bR^{d_{\check\fw}+d_\fv})$ such that $\check\cG H_\al={h_\al}_{|_{\Sigma_{\check\cD}}}$. By Lemma \ref{dominant} (ii),  $|\xi_\fw|$ and $|\xi_{\fv,\fw_0}|$ are bounded up to constants by powers of $|\xi_\fv|$ on~$\Sigma_\cD$. 
By \eqref{check-sigma} and \eqref{gelfand-radon},
we then have that
$$
\tilde h_\al(\xi)=h_\al(\xi_{\check\fw},\xi_{\fv,\fw_0})\in\cS(\Sigma_\cD)\ .
$$

If $F_\al\in\cS(N)^K$ is the function such that $\cG F_\al$ equals $\tilde h_\al$ on $\Sigma_\cD$, it follows from \eqref{gelfand-radon} that $H_\al=\check\cR F_\al$ for every $\al$. This is equivalent to saying that
$\widehat F_\al(v,0,u)=H_\al(v,u)$. By Hadamard's lemma,
$$
\widehat F_\al(v,\zeta,u)-H_\al(v,u)=\sum_{j=1}^{d_{\fw_0}}\zeta_j K_j(v,\zeta,u)\ ,
$$
with $K_j$ smooth. Therefore,
$$
\begin{aligned}
\widehat F(v,\zeta,u)&=\sum_{[\al]= k}\frac1{\al!}(\tilde D^{\al} \widehat F_{\al})(v,\zeta,u)-\sum_{[\al]= k}\sum_{j=1}^{d_{\fw_0}}\frac1{\al!}\zeta_j\tilde D^{\al} K_j(v,\zeta,u) +\sum_{|\gamma'|=k+1}\zeta^{\gamma'} S_{\gamma'}(v,\zeta,u)\\
&=\sum_{[\al]= k}\frac1{\al!}(\tilde D^{\al} \widehat F_{\al})(v,\zeta,u)+\sum_{|\gamma'|=k+1}\zeta^{\gamma'} S'_{\gamma'}(v,\zeta,u)\ .
\end{aligned}
$$

Notice that $S'=\sum_{|\gamma'|=k+1}\zeta^{\gamma'} S'_{\gamma'}\in\cS(\check N\times\fw_0)$ because it is the difference of two Schwartz functions. Since the derivatives of order up to $k$ vanish for $\zeta=0$, it follows from Hadamard's lemma that the same sum $S'$ can be obtained by   replacing each  $S'_{\gamma'}$ by a function $S''_{\gamma'}\in\cS(\check N\times\fw_0)$.

Now we can undo the Fourier transform. In doing so, the monomials in $\zeta$ are turned into derivatives in the $\fw_0$-variables, and each $\tilde D_j$ is turned into the differential operator $\lambda'_{\tilde N}(\rho_j)$, where the symmetrisation is taken on the direct product $\bar N=\check N\times\fw_0$. We denote by $\bar D_j$ this operator.

We then have
$$
F=\sum_{[\al]= k}\frac1{\al!}\bar D^{\al} F_{\al}+\sum_{|\gamma'|=k+1}\de_w^{\gamma'} U_{\gamma'}\ ,
$$
where $F_\al$ satisfies the stated requirements and $U_{\gamma'}\in\cS(\bar N)$ for every $\gamma'$.

It only remains to replace each operator $\bar D^\al$ with the corresponding $D^\al$. In order to do this, it is sufficient to compare the left-invariant vector field $X_{v_0}$ on $N$ corresponding to an element $v_0\in\fv$ with the  left-invariant vector field $\bar X_{v_0}$ on $\bar N$ corresponding to the same $v_0$. Clearly, the difference $X_{v_0}-\bar X_{v_0}$ is a linear combination $\sum_{j=1}^{d_{\fw_0}}\ell_{j,v_0}(v)\de_{w_j}$ where the $\ell_{j,v_0}$ are linear functionals on $\fv$. Therefore, any difference $D^\al-\bar D^\al$ is a sum of terms, each of which contains at least $k+1$ derivatives in the $\fw_0$-variables. The corresponding term $(D^\al-\bar D^\al)F_\al$ is absorbed by the remainder term.

This concludes the proof of Proposition \ref{hadamard}.

\subsection{Proof of Proposition \ref{F-G}}
\label{subsec_pf_propF-G}
Let $(N,K)$ be a nilpotent Gelfand pair satisfying hypothesis (H).
We keep the notation of Definition \ref{def-H} and fix  $F\in\cS(N)^K$.

Consider the restriction of $\cG F$ to the set $\Sigma_{\check\cD}$ in \eqref{check-sigma}, which equals $\check\cG(\check\cR F)$ by \eqref{gelfand-radon}. By Proposition \ref{trivial}, $(\check N,K)$ satisfies property (S). Hence there exists
 $h_0\in\cS(\bR^{d_{\check\fw}+d_\fv})$  extending  $\check\cG(\check\cR F)$. By Lemma \ref{dominant}, $h(\xi)=h_0(\xi_{\check\fw},\xi_\fv)\in\cS(\Sigma_\cD)$. Let $F_0\in\cS(N)^K$ be the function such that $\cG F_0(\xi)=h$.

Then $\check\cR(F-F_0)=0$, which implies that $F-F_0=\sum_{j=1}^{d_{\fw_0}}\de_{w_j}G_j$, with $G_j\in\cS(N)$. By iterated application of Proposition \ref{hadamard}, we find a family $\{F_\al\}_{\al\in\bN^{d_{\fw_0}}\times\bN^{d_{\fv,\fw_0}}}$ such that, for every $k$,
\begin{equation}\label{expansion}
F=\sum_{[\al]\le k}\frac1{\al!}D^\al F_\al+\sum_{|\gamma|=k+1}\de_{w_0}^\gamma R_\gamma\ ,
\end{equation}
and, for every $\al$, $\cG F_\al(\xi)=h_\al(\xi_{\check\fw},\xi_\fv)$, with $h_\al\in\cS(\bR^{d_{\check\fw}+d_\fv})$.

By Whitney's extension theorem \cite{Mal} (see \cite{ADR2} for a proof in the Schwartz setting), there is a function $g\in\cS(\bR^d)$ such that, for every $\al=(\al',\al'')\in\bN^{d_{\fw_0}}\times\bN^{d_{\fv,\fw_0}}$,
$$
\de^{\al'}_{\xi_{\fw_0}}\de^{\al''}_{\xi_{\fv,\fw_0}}(0,\xi_{\check\fw},\xi_\fv,0)=h_\al(\xi_{\check\fw},\xi_\fv)\ .
$$

We take as $G$ the function in $\cS(N)^K$  such that $\cG G=g_{|\Sigma_\cD}$. We must then prove that $G-F\in\cS_0(N)^K$.

Take a monomial $w_0^\beta$   on $\fw_0$. 
Given an integer $k\ge|\beta|$,   decompose $F$ as in \eqref{expansion} and observe that
$$
\int_{\fw_0}\big(G(v,w_0,u)-F(v,w_0,u)\big)w_0^\beta\,dw_0=\int_{\fw_0}\Big(G-\sum_{[\al]\le k}\frac1{\al!}D^\al F_\al\Big)(v,w_0,u)w_0^\beta\,dw_0\ ,
$$
since the remainder term gives integral 0 by integration by parts.

We set
$$
r_k(\xi)=g(\xi)-\sum_{[\al]\le k}\frac1{\al!}h_\al(\xi_{\check\fw},\xi_\fv)\xi_{\fw_0}^{\al'}\xi_{\fv,\fw_0}^{\al''}
$$
so that
$r_k=\cG\Big(G-\sum_{[\al]\le k}\frac1{\al!}D^\al F_\al\Big)$ on $\Sigma_\cD$.

Then Lemma \ref{xi-partition} gives the pointwise identity on $\Sigma_\cD$
$$
\cG\Big(G-\sum_{[\al]\le k}\frac1{\al!}D^\al F_\al\Big)(\xi)=
\sum_{t\in T,\,j\in\bZ}r_k(\xi)\eta_{t,j}(\xi_{\fw_0})\ .
$$

We claim that $k$ can be chosen large enough so that, undoing the Gelfand transform, the series
$$
\sum_{t\in T,\,j\in\bZ}\Big(G-\sum_{[\al]\le k}\frac1{\al!}D^\al F_\al\Big)*_{\fw_0}\psi_{t,j}
$$
 converges  in the $(\cS(N),|\beta|)$-norm  (in which case it converges to $G-\sum_{[\al]\le k}\frac1{\al!}D^\al F_\al$).

By the continuity of $\cG\inv$, there are $m\in\bN$ and $C>0$, depending on $|\beta|$, such that, for every $t,j,k$,
$$
\Big\|\Big(G-\sum_{[\al]\le k}\frac1{\al!}D^\al F_\al\Big)*_{\fw_0}\psi_{t,j}\Big\|_{\cS(N),|\beta|}\le C\|r_k\eta_{t,j}\|_{\cS(\bR^d),m}\  .
$$

Hence we look for $k$ such that $\|r_k\eta_{t,j}\|_{\cS(\bR^d),m}=O(r^{-|j|})$ for every $t\in T$. We first do so with $r_k$ replaced by the remainder $s_{k'}$ in Taylor's formula,
$$
s_{k'}(\xi)=g(\xi)-\sum_{|\al|\le k'}\frac1{\al!}h_\al(\xi_{\check\fw},\xi_\fv)\xi_{\fw_0}^{\al'}\xi_{\fv,\fw_0}^{\al''}\ .
$$

Then, for $\gamma\in\bN^d$ with $|\gamma|\le m$ and for every $M\in\bN$,
$$
\big|\de^\gamma (s_{k'}\eta_{t,j})(\xi)\big|\le C_{m,M}(1+2^{-\tilde mj})(1+|\xi_\fv|)^{-M}\big(|\xi_{\fw_0}|+|\xi_{\fv,\fw_0}|\big)^{k'+1-m}\ ,
$$
where $\tilde m$ only depends on $m$.

From Lemma~\ref{dominant} we obtain that there are positive exponents $a,b,c$, $a',b',c'$ such that, on  $\Sigma_\cD$,
\begin{equation}\label{dominant'}
\text{for $|\xi|$ small: }\begin{cases}|\xi_{\fw_0}|\lesssim |\xi_\fv|^a\\ |\xi_{\fv,\fw_0}|\lesssim |\xi_\fv|^b|\xi_{\fw_0}|^c\ ,\end{cases} \qquad \text{for $|\xi|$ large: }\begin{cases}|\xi_{\fw_0}|\lesssim |\xi_\fv|^{a'}\\ |\xi_{\fv,\fw_0}|\lesssim |\xi_\fv|^{b'}|\xi_{\fw_0}|^{c'}\ .\end{cases} 
\end{equation}

Since the inequalities of Lemma \ref{dominant} remain valid in a $D(\del)$-invariant neighbourhood of $\Sigma_\cD$, we may  assume that \eqref{dominant'} hold uniformly on the support of each $s_{k'}\eta_{t,j}$.

Since there are $\tau,\tau'>0$ such that, on the support of $\eta_{t,j}$, $|\xi_{\fw_0}|\le r^{\tau j}$ if $j\le0$, and $|\xi_{\fw_0}|\le r^{\tau' j}$ if $j>0$, we can choose $k'$ such that  $\|s_{k'}\eta_{t,j}\|_{\cS(\bR^d),m}=O(r^{-|j|})$.

Finally, we choose $k=\max\{[\al]:|\al|\le k'\}$. Since $k>k'$,
$$
r_k-s_{k'}=\sum_{|\al|> k'\,,\,[\al]\le k}\frac1{\al!}h_\al(\xi_{\check\fw},\xi_\fv)\xi_{\fw_0}^{\al'}\xi_{\fv,\fw_0}^{\al''}\ .
$$

The estimates on $\|(r_k-s_{k'})\eta_{t,j}\|_{\cS(\bR^d),m}$ are basically the same.

Now, $\Big(G-\sum_{[\al]\le k}\frac1{\al!}D^\al F_\al\Big)*_{\fw_0}\psi_{t,j}\in\cS_0(N)^K$, because its $\fw_0$-Fourier transform vanishes for $\zeta\in\fw_0$ close to the origin. 
We now  observe that integration against $w_0^\beta$ is a continuous operation in the $(\cS(N),|\beta|)$-norm.
This concludes the proof of Proposition \ref{F-G}.

\subsection{The main theorem}
\label{sec_pf_main_result}

We can finally prove the main results of the paper.
Theorem~\ref{thm_main} is obtained by combining together the partial results stated in
Propositions \ref{S_0} and \ref{F-G}.
It will be used in Section \ref{sec-application} to prove by a recursive argument that the pairs in Table \eqref{table_block1+2} satisfy property (S).

\begin{theorem}
\label{thm_main}
Let $(N,K)$ be a nilpotent Gelfand pair.
We assume that 
\begin{enumerate}
\item[\rm(1)] all the quotient pairs $(N_t,K_t)$ for $t\in \fz_0\backslash\{0\}$ satisfy property (S);
\item[\rm(2)] the pair $(N,K)$ satisfies hypothesis (H).
\end{enumerate}
Then $(N,K)$  satisfies  property (S).
\end{theorem}

\begin{proof}
Given $F\in\cS(N)^K$, let $G$ be as in Proposition \ref{F-G}. By Proposition \ref{S_0}, $\cG(F-G)$ admits a Schwartz extension $h$. Then $g+h$ is a Schwartz extension of $\cG F$. 
Together with \eqref{inclusion}, this shows that $(N,K)$ satisfies property  (S).
\end{proof}

\section{Application to the pairs in Table \eqref{table_block1+2}}
\label{sec-application}

\begin{corollary}\label{cor_thm_main}
Property (S) holds for all the nilpotent Gelfand pairs $(N, K)$
described in Table 
\eqref{table_block1+2}.
\end{corollary}

\begin{proof} We apply inductively  Theorem \ref{thm_main}
to the pairs
in Table 
\eqref{table_block1+2}.
Indeed, the classification of quotient pairs given below shows that this set of pairs  is essentially self-contained, up to the procedures described in Section \ref{sec_procedure} which guarantee hereditarity of property (S);
the lowest-rank pairs that need to be considered in order to start the induction are the following:
\begin{itemize}
 \item line 1: the trivial pair $(\bR,\{1\})$ and $(H_1,{\rm SO_2})$;
 \item line 2: $(H_1,{\rm U}_1)\cong(H_1,{\rm SO_2})$;
 \item line 3: the quaternionic Heisenberg group with Lie algebra $\bH\oplus\IM\bH$, and with ${\rm Sp}_1$ acting nontrivially only on $\fv=\bH$;
 \item lines 4, 5,  6: $(\bC,\{1\})$, $(H_1,{\rm U}_1)$, and $(\bR,\{1\})$, respectively.
 \end{itemize}
 Since in all these cases property (S) is either trivial or proved in \cite{ADR1},  Part (1) of the hypotheses in Theorem \ref{thm_main} is satisfied recursively. 
Part (2) is also satisfied since either $\check N$ is  abelian, in which case we appeal to Proposition \ref{H-checkNabelian}, or $(\check N, K)$ is a (complex or quaternionic) Heisenberg pair 
for which property (S) has been proven in \cite{ADR1}.
 Hypothesis (H) for all these pairs was proved to be satisfied for certain nilpotent Gelfand pairs in \cite{FRY2}.
This yields
 Corollary \ref{cor_thm_main} modulo verification of the list of quotient pairs.
 \end{proof}

In the rest of this section, we list all the quotient pairs of the pairs in Table \eqref{table_block1+2}, in order to verify that they are products of pairs with factors in the same list or for which property~(S) is known.

Let  $I_p$ (resp $0_p$) be the identity (resp. zero) $p\times p$ matrix and
$$
J_p:={\rm diag}(\underbrace{J,J,\dots,J}_{\mbox{$p$ times}})\ .
$$

The Lie bracket $[\ ,\ ]$ is understood as a map from $\fv\times\fv$ to $\fw$, elements of $\bR^n,\,\bC^n$ etc. as column vectors.

\subsubsection{The pair $(\bR^n\oplus \fs\fo_n,{\rm SO}_n)$}\label{line1}
The Lie bracket is
$$
[v,v']=\half(v\,\trans v'-v'\,\trans v)\ .
$$

Up to conjugation by an element of $K$, we may assume that $t\in \fw_0\backslash\{0\}$ has the form
$$
t={\rm diag}(t_1J_{p_1},\dots,t_kJ_{p_k},0_q)\ ,
$$
with $2p_1+\cdots+2p_k+q=n$ and $t_i\ne t_j\ne0$ for every $i\ne j$. Then we have
$$
\begin{aligned}
K_t&={\rm U}_{p_1}\times\cdots\times {\rm U}_{p_k}\times{\rm SO}_q\ ,\\
\fw_t&=\fu_{p_1}\oplus\cdots\oplus \fu_{p_k}\oplus\fs\fo_q\ ,\\
\fn_t&=(\bC^{p_1}\oplus\fu_{p_1})\oplus\cdots\oplus (\bC^{p_1}\oplus\fu_{p_k})\oplus(\bR^q\oplus\fs\fo_q)\ .
\end{aligned}
$$

\subsubsection{The pair $(\bC^n\oplus \fu_n,{\rm U}_n)$}\label{line2}
The Lie bracket is
$$
[v,v']=\half(v{v'}^*-v'v^*)\ .
$$

Up to conjugation by an element of $K$, we may assume that $t\in \fw_0\backslash\{0\}$ has the form
$$
t={\rm diag}(it_1I_{p_1},\dots,it_kI_{p_k})\ ,
$$
with $p_1+\cdots+p_k=n$ and $t_i\ne t_j$ for every $i\ne j$, 
$\tr t=0$,
and we have
$$
\begin{aligned}
K_t&={\rm U}_{p_1}\times\cdots\times {\rm U}_{p_k}\ ,\\
\fw_t&=\fu_{p_1}\oplus\cdots\oplus \fu_{p_k}\ ,\\
\fn_t&=(\bC^{p_1}\oplus\fu_{p_1})\oplus\cdots\oplus (\bC^{p_1}\oplus\fu_{p_k})\ .
\end{aligned}
$$

\subsubsection{The pair $\big(\bH^n\oplus (HS^2_0\bH^n\oplus\IM\bH),{\rm Sp}_n\big)$}\label{line3}
The Lie bracket is
$$
\begin{aligned}[]
[v,v']&=\half\Big(vi{v'}^*-v'iv^*-\frac1n\tr(vi{v'}^*-v'iv^*)I_n\Big)\oplus\IM(v^*v')\\
&=\Big(\half (vi{v'}^*-v'iv^*)-\frac1n\IM\!_i(v^*v')I_n\Big)\oplus\IM(v^*v')\ ,
\end{aligned}
$$
where $\IM\!_i$ denotes the $i$-component of the argument.

Up to conjugation by an element of $K$, we may assume that $t\in \fw_0\backslash\{0\}$ has the form
$$
t={\rm diag}(t_1I_{p_1},\dots,t_kI_{p_k})\ ,
$$
with $p_1+\cdots+p_k=n$ and $t_i\ne t_j$ for every $i\ne j$, 
and we have
$$
\begin{aligned}
K_t&={\rm Sp}_{p_1}\times\cdots\times {\rm Sp}_{p_k}\ ,\\
\fw_t&=(HS^2\bH^{p_1}\oplus\cdots\oplus HS^2\bH^{p_k})_0\oplus\IM\bH\ .
\end{aligned}
$$

Decomposing $v\in\bH^n$ as $v_1\oplus\cdots\oplus v_k$ with $v_j\in\bH^{p_j}$, the Lie bracket in  $\fn_t$  is 
$$
[v,v']_t=\begin{pmatrix}\half(v_1i{v'}_1^*-v'_1iv_1^*)-\frac1n\IM\!_i(v^*v')I_{p_1}&&\\&\ddots&\\&&v_ki{v'}_k^*-v'_kiv_k^*-\frac1n\IM\!_i(v^*v')I_{p_k}\end{pmatrix}\oplus\IM(v^*v')\ .
$$

If we consider, for $1\le j\le k$, the subalgebra $\fh_j$ of $\fn_t$ generated by $\bH^{p_j}$,
$$
\fh_j=\bH^{p_j}\oplus (HS^2\bH^{p_j}\oplus\IM\bH)=\big(\bH^{p_j}\oplus (HS^2_0\bH^{p_j}\oplus\IM\bH)\big)\oplus\bR\ ,
$$
we easily see that it is $K_t$-invariant and  only the factor ${\rm Sp}_{p_j}$ of $K_t$ acts nontrivially on it.  Since $\fh_j$ commutes with $\fh_{j'}$ for $j\ne j'$, it follows that $\fn_t$ is the quotient, modulo a central ideal, of the product of the $\fh_j$.

We conclude that $(\fn_t,K_t)$ is a central reduction of the product of the pairs $(\fh_j,{\rm Sp}_{p_j})$, where, in turn, each $(\fh_j,{\rm Sp}_{p_j})$ is the product of $\big(\bH^{p_j}\oplus (HS^2_0\bH^{p_j}\oplus\IM\bH),{\rm Sp}_{p_j}\big)$ and the trivial pair $(\bR,\{1\})$.

\subsubsection{The pairs $(\bC^{2n+1}\oplus \Lambda^2\bC^{2n+1},{\rm SU}_{2n+1})$ and $\big(\bC^{2n+1}\oplus (\Lambda^2\bC^{2n+1}\oplus\bR)\big),{\rm U}_{2n+1})$}\label{lines45}
To fix the notation, we consider the second family of pairs, the other being analogous and simpler. The Lie bracket is
\begin{equation}\label{bracketline5}
\begin{aligned}[]
[v,v']&=\half(v\trans v'-v'\trans v)\oplus \IM(v^*v')\ .
\end{aligned}
\end{equation}

Up to conjugation by an element of $K$, we may assume that $t\in \fw_0\backslash\{0\}$ has the form
$$
t={\rm diag}(t_1J_{p_1},\dots,t_kJ_{p_k},0_{2q+1})\ ,
$$
with $p_1+\cdots+p_k+q=n$ and $t_i\in\bR$, $t_i\ne t_j\ne0$ for $i\ne j$, we have
$$
\begin{aligned}
K_t&={\rm Sp}_{p_1}\times\cdots\times {\rm Sp}_{p_k}\times {\rm U}_{2q+1}\\
\fw_t&=HS^2\bH^{p_1}J_{p_1}\oplus\cdots\oplus HS^2\bH^{p_k}J_{p_k}\oplus \Lambda^2\bC^{2q+1}\oplus\bR\ .
\end{aligned}
$$

Like in the previous case, we split $\bC^{2n+1}$ as $\bC^{2p_1}\oplus\cdots\oplus \bC^{2p_k}\oplus\bC^{2q+1}$, and set
$$
\fh_j=\bC^{2p_j}\oplus(HS^2\bH^{p_j}J_{p_j}\oplus\bR)\ ,\quad(j=1,\dots,k)\ ,\qquad \fh_{k+1}=\bC^{2q+1}\oplus  (\Lambda^2\bC^{2q+1}\oplus\bR)\ ,
$$
i.e., the subalgebra generated by the $j$-th summand in $\bC^{2n+1}$. Then, for $1\le j\le k$,
$$
\fh_j\cong \bH^{p_j}\oplus(HS^2_0\bH^{p_j}\oplus\bR^2)\ ,
$$
and $(\fh_j,{\rm Sp}_{p_j})$ is isomorphic to a central reduction of the pair in subsection \ref{line3}. Finally, $(\fn_t,K_t)$ is isomorphic to a central reduction of the product of $k$ pairs of this kind and the pair $(\fh_{k+1},{\rm U}_{2q+1})$.

\subsubsection{The pair $\big(\bC^{2n}\oplus (\Lambda^2\bC^{2n}\oplus\bR),{\rm SU}_{2n}\big)$}\label{line6}
The Lie bracket is given by \eqref{bracketline5}. Any element $w_0$ of $\fw_0\backslash\{0\}$  is conjugate, modulo an element of ${\rm U}_{2n}$, to an element of the form
\begin{equation}\label{diag}
t={\rm diag}(t_1J_{p_1},\dots,t_kJ_{p_k},0_{2q})\ ,
\end{equation}
with $p_1+\cdots+p_k+q=n$ and $t_i\in\bR$,  $t_i\ne t_j$ for $i\ne j$. 

If $q>0$, then $w_0$ and $t$ are also conjugate under ${\rm SU}_{2n}$. If $q=0$, then there exists $e^{i\theta}$, unique up to a $2n$-th root of unity, such that $w_0$ is conjugate to 
$$
t_\theta={\rm diag}(t_1e^{i\theta}J_{p_1},\dots,t_ke^{i\theta}J_{p_k})\ .
$$

Then
$$
K_{e^{i\theta}t}=K_t\ ,\qquad
\fw_{e^{i\theta}t}=e^{i\theta}\fw_t\ .
$$

For an element $t$ as in \eqref{diag}, we have
$$
\begin{aligned}
K_t&={\rm Sp}_{p_1}\times\cdots\times {\rm Sp}_{p_k}\times {\rm SU}_{2q}\\
\fw_t&=\begin{cases}HS^2\bH^{p_1}J_{p_1}\oplus\cdots\oplus HS^2\bH^{p_k}J_{p_k}\oplus \Lambda^2\bC^{2q}\oplus\bR &\text{ if } q\ne0\ ,\\
HS^2\bH^{p_1}J_{p_1}\oplus\cdots\oplus HS^2\bH^{p_k}J_{p_k}
\oplus i\bR\, 
{\rm diag}(  t_1^{-1} J_{p_1}, \ldots, t_k^{-1} J_{p_k})
\oplus\bR &\text{ if } q=0\ .
\end{cases}
\end{aligned}
$$

The discussion proceeds as in subsection \ref{lines45}.

\end{document}